\documentclass[12pt]{amsart}

\usepackage{dsfont}
\usepackage{amsmath}
\usepackage{amssymb}
\usepackage{graphicx}

\usepackage{amssymb}

\usepackage{amsfonts}

\usepackage{soul}

\usepackage{mathpazo}
\usepackage{color}
\usepackage{yfonts}

\usepackage{paralist}

\usepackage{stmaryrd}

\usepackage{amsxtra}

\newcommand{\norm}[1]{\| #1\|}

\def \R{{\mathbb R}}

\def \Z{{\mathbb Z}}
\def \N{{\mathbb N}}

\newcommand{\prf}{{\begin{proof}}}
\newcommand{\epf}{{\end{proof}}}

\newtheorem{theo}{\sc Theorem}
\newtheorem{nota}{\sc Notation}

\newtheorem{lemma}{\sc lemma}
\newtheorem{cor}{\sc corollary}

\newtheorem{conj}[theo]{\sc Conjecture}

\theoremstyle{definition}

\def\bee{\begin{equation}}
\def\eee{\end{equation}}

\theoremstyle{rema}

\newcommand{\pdvr}[2]
{\dfrac{\partial^{#2} #1}{\partial \theta^{#2_1} \partial r^{#2_2}}}
\newcommand{\pdvrs}[2]
{\partial^{#2} #1 /\partial \theta^{#2_1} \partial r^{#2_2}}
\newtheorem{thm}{Theorem}
\newtheorem{question}{Question}

\newtheorem{definition}{Definition}
\numberwithin{equation}{section}

\author{Zhiyuan Zhang}

\address{Institut de Math\'ematiques de Jussieu - Paris Rive Gauche, B\^atiment Sophie Germain, Bureau 652, 75205 Paris Cedex 13, France}
\email{zzzhangzhiyuan@gmail.com}

\begin{document}

\title[Finitely Valued Potentials]{On the spectrums of ergodic Schrodinger operators with finitely valued potentials}

\date{\today}

\maketitle

\begin{abstract} 
We show that the Lebesgue measure of the spectrum of ergodic Schr\"odinger operators with potentials defined by non-constant function over any minimal aperiodic finite subshift tends to zero as the coupling constant tends to infinity. We also obtained a quantitative upper bound for the measure of the spectrum. This follows from a result we proved for ergodic Schr\"odinger operators with  potentials generated by aperiodic subshift under two conditions on the recurrence property of the subshift. We also show that one of these conditions is necessary for such result.
\end{abstract}

\tableofcontents
\addtocontents{toc}{\protect\setcounter{tocdepth}{1}}


\section{Introduction}

This paper is motivated by Simon's subshift conjecture ( in \cite{S}, see also \cite{D} ) and the desire to get a better understanding of recently discovered counter-examples in \cite{ADZ}. Consider an aperiodic strictly ergodic subshift over a finite alphabet, which is assumed to consist of real numbers for simplicity, consider the Schr\"odinger operators in $\ell^2(\Z)$ with potentials given by the elements of the subshift. By minimality, the spectrum is the same for every element in the subshift. The common spectrum was suspected to be of zero Lebesgue measure. For CMV matrices, Barry Simon conjectured the following in \cite{S}.

\begin{conj}
Given a minimal subshift of Verblunsky coefficients which is not periodic, the common essential support of the associated measures has zero Lebesgue measure.
\end{conj}

 There is also a Schr\"odinger version of the subshift conjecture ( see \cite{ADZ} ),
 \begin{conj}
 Given $\mathcal{A} \subset \R$ finite and a minimal subshift $\Omega \subset \mathcal{A}^{\Z}$ which is not periodic, the associated common spectrum has zero Lebesgue measure.
 \end{conj}
 
  It has been shown that for strictly ergodic subshifts satisfying the so-called Boshernitzan condition, the Schr\"odinger operators have zero-measure spectrum for any non-constant potentials \cite{DL}, and for CMV matrices, one has zero-measure supports \cite{DL2}. More results on subshifts associated operators can be found in \cite{D}.

In the recent work of Avila, Damanik and Zhang \cite{ADZ}, the subshift conjecture is shown to be false, for both Schr\"odinger version and the orginal version for CMV matrices. In fact, the authors proved the following theorem for Schr\"odinger operators ( Theorem 1 in \cite{ADZ} ) 
\begin{thm} \label{1}
Given $\mathcal{A} \subset \R$ with $2\leq card \mathcal{A} < \infty$, there is a minimal subshift $\Omega \subset \mathcal{A}^{\Z}$ which is not periodic, such that the associated spectrum $\Sigma \subset \R$ has strictly positive Lebesgue measure.
\end{thm}
They also proved a CMV matrices analog ( Theorem 2 in \cite{ADZ} ) which disproved the subshift conjecture in its original formulation.

In \cite{ADZ}, the authors also proved a positive result roughly saying that when the system endowed with an ergodic invariant measure is relatively simple, the associated density of states measure is purely singular. The precise condition is formulated as being "almost surely polynomially transitive" and "almost surely of polynomial complexity".  This theorem works for subshifts generated by translations on tori with Diophantine frequencies, certain skew shifts and interval exchange transformations. Note that this theorem does not imply that the measure of the spectrum is zero.

Given this new phenomenon,  namely that subshift generated potentials can give positive-measure spectrum, the following question arises naturally.

\begin{question}
Given a minimal aperiodic subshift and a non-constant potential function,
how large can the Lebesgue measure of the spectrum be ? 
\end{question}

This paper is an attempt to study this question. The main result is the following.

\begin{thm} \label{minimal cor}
Given any $k \geq 2$, a minimal aperiodic subshift $\Omega \subset \{1, \cdots, k\}^{\Z}$. Then for any $0 < \gamma < \frac{1}{4}$ the following is true. For any non-constant function $v: \{1,\cdots,k\}\to \R$, there exists $C>0$, such that for any $\lambda > 0$, the Lebesgue measure of the spectrum of the Schr\"odinger operator with potential $\lambda v$ is smaller than $C\lambda^{-\gamma}$.
\end{thm}

We actually proved the following more general result for ergodic Schr\"odinger operators with shift-generated potentials

\begin{thm} \label{main theorem}
Given any $k\geq 2$, an aperiodic subshift $\Omega \subset \{1,\cdots, k\}^{\Z}$ endowed with an ergodic shift invariant measure $\mu$, such that : (1) there exists an integer $K > 0$ such that $\mu(\{ \omega ; \omega_0 = \omega_1 = \cdots = \omega_{K-1} \}) = 0$; (2) there exists an integer $L > 0$ such that for any $1\leq i \leq k$, any $\omega = ( \omega_p)_{p \in \Z} \in supp \mu$, there exists $0 \leq j \leq L-1$ such that $\omega_j = i$. Then for any $0 < \gamma < \frac{1}{4}$, there exists a constant $C > 0$, such that for any non-constant function $v : \{1, \cdots, k\} \to \R$, denote $\lambda = \min(|v(i)-v(j)|; 1\leq i < j \leq k)$, then $Leb(\Sigma_{v}) < C \lambda^{-\gamma}$. Here $\Sigma_{v}$ denotes the almost sure spectrum with potential $v$.
\end{thm}

In fact, we will prove a better bound for the exponent $\gamma$ based on more detailed knowledge of the recurrence property of the subshift.

Since any minimal subshift $\Omega$, any ergodic shift invariant measure $\mu$ on $\Omega$ satisfy condition (1),(2) in Theorem \ref{main theorem}, Theorem \ref{minimal cor} follows as an immediate corollary.

To the best of the author's knowledge, this result seems to be the first non-trivial upper bound for the Lebesgue measure of the spectrum for this class of Schr\"odinger operators without any complexity bound assumption.

We note that if one only assumes the conditions of Theorem \ref{main theorem}, one cannot hope to prove zero-measure spectrum for all sufficiently sparse potentials. In fact we have the following theorem which is a slight modification of Theorem 1 in \cite{ADZ}.

\begin{thm} \label{positive measure spectrum for large couplings theorem}
Given any $k \geq 2$, $\epsilon > 0$, any countable subset $B$ of non-constant functions from $\mathcal{A}$ to $\R$.There exists $C > 0$, a minimal aperiodic subshift $\Omega \subset \{1,\cdots, k\}^{\Z}$ with complexity function $p$ satisfying $p(n) < Cn^{1 + \epsilon}, \forall n \in \N$, such that for any $v \in B$,the Schr\"odinger operator with potential $v$ has spectrum of strictly positive Lebesgue measure.
\end{thm}

Here for any $n \geq 1$, the complexity function $p(n)$ denote the number of different words of length $n$ appeared in the subshift. This notion can also be found in many literatures on Schr\"odinger operators with shift-generated potentials, for example \cite{ADZ}, \cite{D} and \cite{DL}.

We also note that the condition (1) in Theorem \ref{main theorem} is necessary to ensure that the measure of the spectrum tends to zero as the "sparseness" of the potential function grows to infinity. This is seen from the following theorem, which seems to be folklore.

\begin{thm}\label{unbounded repetitions theorem}
Given any $k\geq 2$, a subshift $\Omega \subset \{1, \cdots, k\}^{\Z}$, an ergodic shift invariant measure $\mu$ such that there exists $i \in \{1,\cdots, k\}$ such that for any integer $N>0$, $\mu( \{ \omega ; \omega_0 = \omega_1 = \cdots = \omega_{N-1} = i \}) > 0$. Then for any function $v: \{1, \cdots, k\} \to \R$, we have $[-2+v(i),2+v(i)] \subset \Sigma_{v}$. Here $\Sigma_{v}$ denotes the almost sure spectrum with potential $v$.
\end{thm}

\subsection{Outline of the proof}

As mentioned above, the subshift conjecture is true for many subshifts. As discussed in \cite{D}, two principal approachs for establishing zero-measure spectrum are: 1. Using trace map dynamics; 2.Proving uniform convergence, usually under Boshernitzan's condition. In both cases, one first show that the spectrum coincides with the set of energy on which the Lyapunov exponent vanish, then apply Kotani's theory \cite{K}. Thus in these approaches, one comes down to showing that non-uniformly hyperbolicity does not appear at all. 

In order to prove our result, we have to consider the possible appearance of non-uniformly hyperbolic dynamics.Then the main task is to show that the set of energy corresponding to non-uniformly hyperbolic dynamics has small measure. Instead of directly establishing uniformly hyperbolicity for many energies, we appeal to Berezansky's theorem in the spectral theory of lattice Schr\"odinger operators, which says that for almost every energy with respect to the spectral measure, there exists a generalised eigenfunction with polynomial growth. We will construct a closed subset $J \subset \R$ of small Lebesgue measure and a subset $\Omega'$ of the shift space of positive measure, such that for element $\omega \in \Omega'$, for any energy outside of this closed set, the Schr\"odinger operator associated to $\omega$ has no generalised eigenfunction of polynomial growth. This approach concerning the generalised eigenfunction is inspired by the proof of Theorem 3 in \cite{ADZ}.

The main technical difficulty with this naive approach is that : We still have to consider dynamics associated with different energy, whose longtime behaviours could be very different. We overcome this difficulty using the so-called Benedicks-Carleson argument that is originated in the study of H\'enon maps.  It was introduced to the study of quasi-periodic cocycles by Young \cite{Y}, who showed among other things that for certain parametrised family of quasi-periodic cocycles, the Lyapunov exponents are large for a large set of parameters. More recent developments of this type of arguments can by found in \cite{Bj},\cite{WZ}. Our main observation is that Benedicks-Carleson arguments provide a unified mechanism for hyperbolicity for all the energy that is not removed from the parameter exclusion. Roughly speaking, for a short interval of energy that could cause non-uniformly hyperbolicity, we have only one "bad" alphabet that could ruin the exponential growth of the associated cocycle. We inductively define a nested sequence of subset of the subshift starting this alphabet, so that $(n+1)-$th set is contained in $n-$th set,  and each time we consider the Poincar\'e return map restricted to $(n+1)-$set and form an accelerated cocycle defined over $(n+1)-$th set, which is just the consecutive multiplication along the first return map. We inductively prove that the accelerated cocycles are highly hyperbolic and the most expanding and most contracting directions can be related to those of the previous accelerated cocycles. The only problem occurs when apply the matrix corresponding to the "bad" alphabet. We then remove a set of energy each time to produce certain amount of transversality. For the remains of energies, the corresponding Schr\"odinger cocycles are exponentially increasing along a subsequence in time ( this can be compared to one of the main results in \cite{Yoc}, which says that a cocycle is uniformly exponentially increasing is equivalently to being uniformly hyperbolic ). Since we can get good control of the closeness of the stable/unstable directions for matrices in consecutive steps, the parameter removed in each step stays close to the parameters removed in the previous step. Finally, we find a subset of the subshift with positive measure whose elements have good forward and backward landing time at arbitrarily large time scale, which would preclude the existence of generalised eigenfunctions of polynomial growth.

\subsection{Structure of the paper}
In Section \ref{preliminary and notations}, we introduce the setting and the notations. We also show an \textit{a priori} bound for the spectrum based on a classical theorem of Johnson. In Section \ref{a tower construction}, we introduce a sequence of objets and parameters that will later help us estimate the spectrum and control the dynamics. In Section \ref{an iteration scheme} we deal with a technical lemma that will be used repeatedly in Section \ref{choosing the parameters}. Section \ref{choosing the parameters} is devoted to the construction and estimation of the objets introduced in Section \ref{a tower construction}. In Section \ref{cover the spectrum}, we relate the objets introduced in Section \ref{a tower construction} to the spectrum, which is the main novelty of this paper. In Section \ref{Area of the spectrum and the proof of the main theorem}, we estimate the spectrum and conclude the proof of the main theorem. In Section \ref{folklore theorems}, we prove Theorem \ref{positive measure spectrum for large couplings theorem} and Theorem \ref{unbounded repetitions theorem}.

\subsection*{Acknowledgement}
I am grateful to Artur Avila for his supervision. I thank Jean-Paul Allouche for his comments on complexity functions which were used in an earlier version. I thank S\'ebastien Gou\"ezel for useful conversations which give a part of the motivation of this paper.
I thank Zhenghe Zhang for reading an earlier version of this paper and comments. Special thanks go to David Damanik for his generous encouragement, his interest in this problem and detailed comments, these including pointing out an important mistake in the statement of Theorem \ref{minimal cor} in an earlier version; and to Qi Zhou for his consistent support and many interesting mathematical and non-mathematical conversations.


\section{Ergodic Schr\"odinger operators over subshifts} \label{preliminary and notations}

Given a finite set $\mathcal{A}$, 
we define the shift transformation $T$ on $\mathcal{A}^{\Z}$ by $T(\omega)_{n} = \omega_{n+1}$. 
Let  $\Omega$ be a $T-$invariant compact subset of $\mathcal{A}^{\Z}$. Let $\mu \in \mathcal{P}(\Omega)$ be an ergodic $T-$invariant measure.  Without loss of generality, in this paper we will always assume that
\begin{eqnarray*}
\Omega = supp \mu
\end{eqnarray*}
for otherwise we can replace $\Omega$ by $supp \mu$.
We will assume that for any $\alpha \in \mathcal{A}$, we have
\begin{eqnarray*}
\mu(\{\omega; \omega_0 = \alpha\}) > 0
\end{eqnarray*}
for otherwise we can replace $\mathcal{A}$ by one of its subsets.

Let $v : \mathcal{A} \to \R$ be a function. Without loss of generality, in this paper we will always assume that: for any $\alpha, \beta \in \mathcal{A}$, we have $v(\alpha) \neq v(\beta)$. To each such $v$, we can associate a continuous function $V : \Omega \to \R$ defined by $V(\omega) = v(\omega_0)$. In the study of ergodic Schr\"odinger operators, $V$ is usually referred to as the potential function. In the following, we will call both $V$ and $v$ the potential without causing ambiguity in understand the results. For each $\omega \in \Omega$, let $\Sigma_{\omega}$ denote the spectrum of the Schr\"odinger operator $H_{\omega}$ on $\ell^{2}(\Z)$ defined by 
\begin{eqnarray}
(H_{\omega}u)_n = u_{n+1} + u_{n-1} + V(T^{n}\omega) u_n 
\end{eqnarray}
It is well-known that $\Sigma_{\omega}$ is the same for $\mu$ almost every $\omega$. 
 We denote the almost sure spectrum of this family of operators by $\Sigma_{v}$. When there is no confusion on the potential function $v$, we denote $\Sigma = \Sigma_{v}$.  It is also well-known that when $(\Omega, T)$ is minimal, $\Sigma_{\omega}$ is the same for all $\omega$. Although we will not exploit this fact in this paper.

Denote $R = \{ v(\alpha) \}_{\alpha \in \mathcal{A}}$.
For any $\alpha \in \mathcal{A}$, we denote
\begin{eqnarray*}
   A^{E}_{\alpha} & = & \begin{bmatrix}
                                      E- v({\alpha}) & -1 \\
                                      1                     &  0 
                              \end{bmatrix} 
\end{eqnarray*}
and
\begin{eqnarray*}
 A^{E}(\omega) &=& A^{E}_{\omega_0}
\end{eqnarray*}
We define a function $A^{E} : \Z \times \Omega \to SL(2,\R)$ by setting
\begin{eqnarray*}
A^{E}(0, \omega) &=& Id \\
A^{E}(k, \omega) &=& A^{E}(T^{k-1}(\omega))\cdots A^{E}(\omega)   \mbox{   for all $k > 0$}
\end{eqnarray*}
and
\begin{eqnarray*}
A^{E}(-k, \omega) &=& A^{E}(T^{-k}(\omega))^{-1} \cdots A^{E}(T^{-1}(\omega))^{-1}   \mbox{   for all $k > 0$}
\end{eqnarray*}
For any $n,m \geq 0$, any $\omega \in \Omega$ we have the following relation
\begin{eqnarray*}
A^{E}(n+m, \omega) = A^{E}(m, T^{n}(\omega)) A^{E}(n, \omega)
\end{eqnarray*}
For any finite word $\alpha = \omega_0 \omega_1 \cdots \omega_{n-1}$, where $\omega_i \in \mathcal{A}$ for all $0\leq i \leq n-1$, we define
\begin{eqnarray*}
A^{E}(\alpha) = A^{E}_{\omega_{n-1}}\cdots A^{E}_{\omega_0}
\end{eqnarray*}

\begin{definition}
For any function $v : \mathcal{A} \to \R$, we call $v$ an  \textnormal{ admissible potential } if for any two distinct elements $\alpha,\beta\in \mathcal{A}$, we have $v(\alpha) \neq v(\beta)$. For any admissible potential $v$,
we denote $\lambda_{v} = \min_{\alpha,\beta \in \mathcal{A}, \alpha \neq \beta}|v(\alpha) - v(\beta)|$, and call it the \textnormal{sparseness constant} of the potential $v$.
\end{definition}

We have the following notion called "Uniformly Hyperbolic". We use the definition in \cite{Yoc}, adapted to our situation.
\begin{definition}
Fix an admissible potential $v : \mathcal{A} \to \R$, for each $E \in \R$, we have a map  $A^{E}(1,\cdot): \Omega \to SL(2,\R)$, and we call it the  \textnormal{Schr\"odinger cocycle at energy $E$}. The Schr\"odinger cocycle at energy $E$ is called  \textnormal{Uniformly Hyperbolic} if there exists two (necessarily unique) invariant continuous sections
\begin{eqnarray*}
e_s, e_u : \Omega \to \mathbb{P}^{1}\R^2
\end{eqnarray*}
with $e_s(\omega) \neq e_u(\omega)$ for any $\omega \in \Omega$, and $e_s$ is uniformly repelling ( in the $\mathbb{P}^{1}\R^2$ direction) and $e_u$ is uniformly contracting (in the $\mathbb{P}^{1}\R^2$ direction).
\end{definition}

We have the following well-known result (see \cite{J})
\begin{thm}[Johnson]\label{johnson}
We have $\Sigma = \{E; \mbox{ $A^{E}(1,\cdot)$ is not Uniformly Hyperbolic}\}$
\end{thm}

For any $E_0 \in  R$, we associate an interval centered at $E_0$ $$I_{E_0} = [E_0-H, E_0+H]$$ for some constant $H > 0$ to be determined as follow.

We choose $H > 0$ such that, for any $E \notin \bigcup_{E_0 \in  R} I_{E_0}$, $A^{E}(1,\cdot)$ is Uniformly Hyperbolic.
Indeed, when $H$ is sufficiently large, for any $E \notin \bigcup_{E_0 \in  R} I_{E_0}$, there exist two closed cones $C_{+},C_{-} \subset \R^2$ such that for any $\alpha \in \mathcal{A}$, we have $A^{E}_{\alpha}(C_{+}) \setminus \{0\} \subset int C_{+}$ and $(A^{E}_{\alpha})^{-1}(C_{-}) \setminus \{0\} \subset int C_{-}$.
A classical construction in dynamical systems shows that this implies $A^{E}(1,\cdot)$ is Uniformly Hyperbolic.

Hence by Theorem \ref{johnson} 
\begin{eqnarray}\label{sigma not in u h}
\Sigma \subset \bigcup_{E_0 \in  R} I_{E_0} 
\end{eqnarray}

We will need the following general result on lattice Schr\"odginer operators.( see \cite{BSU})
\begin{thm}[Berezansky]\label{Berezansky}
Almost every $E$ with respect to the spectral measure admits a generalized eigenfunction of polynomial growth.
\end{thm}
In particular, Theorem \ref{Berezansky} implies that for any potential $v$, for any $\omega \in \Omega$, almost every $E$ with respect to the spectral measure of the Schr\"odinger operator associated to $\omega$, there exists $X \in \R^2$, $C,d > 0$ such that
\begin{eqnarray*}
\norm{A^{E}(n, \omega)X} \leq C(|n|+1)^{d}, \forall n \in \Z
\end{eqnarray*}

\subsection{Notations}

Throughout this paper, we will use $\lesssim$ and $\gtrsim$ to denote less than or greater than up to multiplying a universal constant. In places we use Laudau's $O(f)$ to denote a quantity majorized by a universal constant times $f$, and use $\Theta(f)$ to denote a quantity minorized by a positive universal constant times $f$.

For any $a,b \in \R$, we will use $|a-b|_{\R/\pi\Z}$ to denote the distance from $a-b$ to the set $\{k\pi\}_{k \in \Z}$. For any two vectors $X_1,X_2 \in \R^2$ such that $X_i = r_i \begin{bmatrix} \cos\theta_i \\ \sin\theta_i \end{bmatrix}$ for $i=1,2$, we denote $\angle(X_1,X_2) = |\theta_1-\theta_2|_{\R/\pi\Z}$.


\section{A tower construction} \label{a tower construction}

In order to prove Theorem \ref{main theorem}, it suffices to prove that for any $\alpha \in \mathcal{A}$, we have the corresponding upper bound for the Lebesgue measure of $\Sigma \bigcap I_{v(\alpha)}$. Then Theorem \ref{main theorem} will follow from \eqref{sigma not in u h} and the fact that $card(A) < \infty$.

Throughout Section \ref{a tower construction} to Section \ref{Area of the spectrum and the proof of the main theorem}, we fix $\alpha_0 \in \mathcal{A}$ and denote $E_0 =  v(\alpha_0) \in R$.
Then Theorem \ref{main theorem} is reduced to the following.
\begin{thm} \label{real main theorem}
Under the condition of Theorem \ref{main theorem}, for any $0 < \gamma < \frac{1}{4}$ where $L$ is given by condition (2) in Theorem \ref{main theorem}, there exists a constant $Q > 0$, such that for any admissible potential $v : \{1,\cdots,k\} \to \R$, we have $Leb(\Sigma_v \bigcap I_{E_0})  < Q\lambda_{v}^{-\gamma}$.
\end{thm}

Hereafter, we will assume that the condition in Theorem \ref{main theorem} holds. We denote
\begin{eqnarray*}
\lambda &=& \lambda_{v}
\end{eqnarray*}

Define 
\begin{eqnarray*} 
\Delta & = & \{ \omega \in \Omega; \omega_{-1} \neq \alpha_0, \omega_{0} =\alpha_0 \}  \\
\Delta_{(i)} & = & \{ \omega \in \Omega; \omega_{-1} \neq \alpha_0, \omega_{0} =\alpha_0, \cdots, \omega_{i-1} =\alpha_0, \omega_{i} \neq \alpha_0  \}  
\end{eqnarray*}

Since ergodic subshift $(\Omega, T, \mu)$ satisfies the condition (1) in Theorem \ref{main theorem}, then there exists $K > 0$ such that
\begin{eqnarray*}
 \Delta = \bigsqcup _{i = 1}^{ K} \Delta_{(i)}  \mbox{   up to a $\mu-$null set}
\end{eqnarray*}
We define 
\begin{eqnarray*}
\Delta_0 &=&\Delta  
\end{eqnarray*}
By our assumptions in Section \ref{preliminary and notations}, we have
\begin{eqnarray*}
\mu(\Delta_0) > 0
\end{eqnarray*}

After possibly removing a $\mu-$null set from $\Delta_0$, we can assume that 
for any $\omega \in \Delta_0$, there exist integers $n, m >0$ such that $T^{n}(\omega) \in \Delta_0$ and $T^{-m}(\omega) \in \Delta_0$. 

For any $E \in \R$, for any $\omega \in \Delta_0$,  we define
\begin{eqnarray*}
l_0(\omega) &=& \inf\{ k ; k > 0 , T^k(\omega) \in \Delta_0 \}  \\
T_0(\omega) &=& T^{l_0(\omega)}(\omega) \\
A_0^{E}(\omega) &=& A^{E}(l_0(\omega), \omega)
\end{eqnarray*}
Note that there is an ergodic $T_0-$invariant probability measure $\mu_0$ on $\Delta_0$ given by
$$\mu_0 = \frac{1}{\Delta_0}\mu|_{\Delta_0}$$

For any $E \in \R$, we denote $C^{E} = A^{E}_{\alpha_0}$ and $C^{E}_i = (C^{E})^{i}$.
For any $\omega \in \Delta_{(i)}$, we have that 
$$A^{E}(T^{k}\omega) = C^{E}    \mbox{ for all $ 0 \leq k \leq i-1 $ } $$
For any $1 \leq i \leq K$, for all $\omega \in \Delta_{(i)}$, we define
\begin{eqnarray*}
C^{E}(\omega) &=& A^{E}(i, \omega) = C^{E}_i \\
B_0^{E}(\omega) &=& A^{E}(l_0(\omega)-i, T^{i}(\omega)) = A_0^{E}(\omega)(C^{E}(\omega))^{-1}
\end{eqnarray*}

In the following, for any $n\geq 0$, we are going to define $\Delta_n \subset \Delta_0$, to which we associate a map $T_n: \Delta_n \to \Delta_n$, an ergodic $T_n$-invariant probability measure $\mu_n$, functions $l_n: \Delta_n \to \Z$, $r_n : \Delta_{n+1} \to \Z$,
 $A^{E}_n,B^{E}_n : \Delta_n \to SL(2,\R)$ satisfying the following properties:
 
 (P1) For any $ n \geq 0$, $\mu(\Delta_n) > 0$ and $\Delta_{n+1} \subset \Delta_{n}$;
 
 (P2) For every $\omega \in \Delta_n$, $l_n(\omega) = \inf\{ m > 0; T^{m}(\omega) \in \Delta_n \} < \infty$ and $T_n(\omega) = T^{l_n(\omega)}(\omega)$;
 
 (P3) $\mu_n = \frac{1}{\mu(\Delta_n)}\mu|_{\Delta_{n}}$;

 (P4) $A^{E}_n(\omega) = A^{E}(l_n(\omega), \omega)$;
 
 (P5) For each $1\leq i\leq K$, for any $\omega \in \Delta_n \cap \Delta_{(i)}$, we have $A^{E}_n(\omega) = B^{E}_n(\omega) C^{E}_i$.

(P6) $r_n(\omega) = \inf\{ k ; k>0, T_n^k(\omega) \in \Delta_{n+1} \} < \infty$  for all $\omega \in \Delta_{n+1}$.

By (P2), we see that $T_n$ is the Poincar\'e return map on $\Delta_n$.

By (P2),(P4),(P5) and (P6) we get
\begin{eqnarray}
\label{used once no}l_{n+1}(\omega) &=& \sum_{i=0}^{r_n(\omega)-1} l_n(T_n^{i}\omega),\forall n \geq 0, \forall \omega \in \Delta_{n+1} \\
B_n^{E}(\omega) &=& A^{E}(l_n(\omega) - i, T^{i}(\omega)),\forall n\geq0, \forall \omega \in \Delta_{n} \bigcap \Delta_{(i)}
\end{eqnarray}

By the definition of $\Delta_0, \Delta_{(i)}$, we see that $l_0(\omega) \geq 1$ for all $\omega \in \Delta_0$.
Hence by \eqref{used once no}, we have $l_n(\omega) \geq 1$ for all $n \geq 0$ and $\omega \in \Delta_{n}$.

\begin{nota} \label{conventions}
For any matrix $A \in SL(2,\R)\setminus SO(2,\R)$, we denote $u(A),s(A),\lambda(A)$ to be real numbers that satisfy \begin{eqnarray*}
A &=& R_{u(A)}\begin{bmatrix} \lambda(A) & 0 \\ 0 & \lambda(A)^{-1} \end{bmatrix} R_{\frac{\pi}{2} - s(A)} \\
\mbox{and } \lambda(A) &>& 1
\end{eqnarray*}
Here $u(A),s(A)$ are well-defined up to adding a multiple of $\pi$.
\end{nota}

In Section \ref{choosing the parameters} we will construct a finite union of intervals, denoted by $J_n \subset I_{E_0}$ for each $n\geq 0$.
We now introduce a sequence of parameters $\bar{\lambda}_n, \zeta_n, \chi_n, M_n, N_n, \kappa_n > 0$ satisfying the following estimates: 
\begin{eqnarray}
\label{ratios of time intervals}   0 < \sup_{\omega \in \Delta_n} l_n(\omega) &\leq& M_n \inf_{\omega \in \Delta_n} l_n(\omega) \\
\label{r n well controlled} r_n(\omega) &\in& \{N_n, N_n+1\} ,\forall n \geq 0, \forall \omega \in \Delta_{n+1}\\
\label{N n inverse summable}\sum_{n = 0}^{\infty} \frac{1}{N_n} &<& \infty
\end{eqnarray}

For any $n\geq 0$, any $E \in I_{E_0} \setminus \bigcup_{n\geq m\geq 0}J_m$, any $\omega \in \Delta_{n+1}$, any $0 \leq q < r \leq r_n(\omega)$, denote
\begin{eqnarray*}
B^{E} = A^{E}_n(T^{r-1}_n(\omega)) \cdots A^{E}_n(T_n^{q+1}(\omega))B^{E}_n(T_n^q(\omega))
\end{eqnarray*}
then
\begin{eqnarray}
\label{B E hyperbolic} B^{E} \in SL(2,\R) \setminus SO(2,\R)
\end{eqnarray}
and
\begin{eqnarray}
\label{C1 bound of u n}   |u(B^{E})  - u(B^{E}_n(T_{n}^{r-1}(\omega)))|_{\R/\pi\Z} &\leq& \zeta_n \\
\label{C1 bound of s n}   |s(B^{E})  - s(B^{E}_n(T_n^{q}(\omega)))|_{\R/\pi\Z} &\leq& \zeta_n    \\
\label{lower bound of norm 2} \lambda(B^{E}) \geq  e^{\chi_{n+1} \sum_{i=q}^{r-1}l_n(T_n^{i}(\omega))}
\end{eqnarray}
Note that by taking $r=q+1$ and \eqref{B E hyperbolic}, we have $B^{E}_n(T_n^{r-1}(\omega)), B^{E}_n(T_n^{q}(\omega)) \in SL(2,\R) \setminus SO(2,\R)$. This shows that the left hand side of \eqref{C1 bound of u n} and \eqref{C1 bound of s n} are well-defined.

Moreover, for any $n \geq 0$, any $E \in I_{E_0} \setminus \bigcup_{n-1 \geq m \geq 0} J_{m}$, any $\omega \in \Delta_n$, we have
\begin{eqnarray}
\label{lower bound of norm}   \lambda(B^{E}_n(\omega)) \geq e^{\chi_n l_n(\omega)} \geq \bar{\lambda}_n
\end{eqnarray}

We will choose an absolute constant $C > 1$ such that 
\begin{eqnarray}\label{C1 upper bound of critical matrices} \norm{C^{E}_k},\norm{\partial_E C^{E}_k} \leq C, \forall 1 \leq k \leq K, \forall E \in I_{E_0} \end{eqnarray}

We will use the following lemma to determine the values of $\bar{\lambda}_0, \chi_0, M_0$.

\begin{lemma}\label{lemma initial data}
We can choose $M_0 = \frac{\sup_{\omega \in \Delta_0} l_0(\omega) }{ \inf_{\omega \in \Delta_0} l_0(\omega) } < \infty$ so that \eqref{ratios of time intervals} is valid for $n=0$. For any $\lambda > 0$ sufficiently large, we can choose $\bar{\lambda}_0 = \frac{1}{2}\lambda$,$ \chi_0 = \log \bar{\lambda}_0$ so that \eqref{lower bound of norm} is valid for $n=0$. Moreover, for any $\lambda > 0$ sufficiently large, we have \eqref{B E hyperbolic} for $n=0$.
\end{lemma}
\begin{proof}
The hypothesis \eqref{ratios of time intervals} follows from the definition of $M_0$. It follows from condition (2) in Theorem \ref{main theorem} that $M_0 < \infty$.

For all $\lambda$ sufficient large, for any $E \in I_{E_0}$, any $\alpha \in \mathcal{A}$ distinct from $\alpha_0$,
$A^{E}_{\alpha} = \begin{bmatrix} \eta & -1 \\ 1 & 0 \end{bmatrix}$ with $\eta = E-v({\alpha})$.
When $\lambda$ is sufficiently large, we have $|E-v(\alpha)| \geq |v(\alpha_0) - v({\alpha})| - | v(\alpha_0) - E| \geq \lambda - H > \frac{9}{10}\lambda$.
It is direct to check that there exists absolute constants $\epsilon > 0$, $\Lambda >0$, such that the following is true. Denote $\mathcal{C} \subset \R^{2} \setminus \{(0,0)\}$ as 
\begin{eqnarray*}
\mathcal{C} = \{(x,y) ; x\neq 0, |y| \leq \epsilon |x|\}
\end{eqnarray*}
for any $X \in \mathcal{C}$, for any $\eta$ such that $|\eta| > \Lambda$ we have
\begin{eqnarray*}
 \begin{bmatrix} \eta & -1 \\ 1 & 0 \end{bmatrix}  X \in \mathcal{C}
 \mbox{ and}  \norm{\begin{bmatrix} \eta & -1 \\ 1 & 0 \end{bmatrix}  X} \geq \frac{2}{3}\eta\norm{X}
\end{eqnarray*}
Then when $\lambda$ is sufficiently large, for any $\alpha \in \mathcal{A}$ distinct from $\alpha_0$, any $E\in I_{E_0}$ and any $X\in \mathcal{C}$, we have
\begin{eqnarray*}
A^{E}_{\alpha} X \in \mathcal{C}
 \mbox{ and}  \norm{A^{E}_{\alpha} X} \geq \frac{9}{10}\times\frac{2}{3}\lambda\norm{X} \geq \frac{1}{2}\lambda\norm{X}
\end{eqnarray*}

Since for any $\omega \in \Delta_0$,  $B^{E}$ in \eqref{B E hyperbolic} is a product of some matrices in set $\{A^{E}_{\alpha}\}_{\alpha \neq \alpha_0}$, we have \eqref{lower bound of norm}, \eqref{B E hyperbolic} for $n=0$ with our choices of $\bar{\lambda}_0, \chi_0$ in the statement. This completes the proof.
\end{proof}

The sets $J_n$ will be defined and the precise choices of parameters $ \bar{\lambda}_n, \zeta_n, \chi_n, M_n, N_n, \kappa_n $ will be made clear in Section \ref{choosing the parameters}.

We have the following lemma that will be used repeatedly.
\begin{lemma} \label{C0 close then angle close}
There exists $c_5 > 0$ such that for any $\epsilon > 0$, any $u,s,\tilde{u},\tilde{s} \in \R$ satisfying $|u-\tilde{u}|_{\R/\pi\Z},|s-\tilde{s}|_{\R/\pi\Z} < \epsilon$, for any $1\leq k \leq K$, any $E \in I_{E_0}$, we have
$$\angle( R_{\frac{\pi}{2} - \tilde{s}} C^{E}_k R_{\tilde{u}} \begin{bmatrix} 1 \\ 0\end{bmatrix}, R_{\frac{\pi}{2} - s} C^{E}_k R_{u} \begin{bmatrix} 1 \\ 0\end{bmatrix} ) < c_5 \epsilon$$
\end{lemma}
\begin{proof}
Since the norm of $C^{E}_k$ is uniformly bounded for all $E \in I_{E_0}$ and $1\leq k\leq K$,
the lemma follows from straight-forward calculations.
\end{proof}


\section{An iteration scheme} \label{an iteration scheme}

In this section, we will prove a lemma that will help us control the dynamics for energies that satisfy certain transversality condition.
Throughout this section, we will use the following notations.
\begin{nota} \label{notation 2}
    For any $E \in I_{E_0}$, any $n\geq 0$, any $\omega \in \Delta_n$, integer $r\geq 1$ such that 
    \begin{eqnarray}\label{the last label}
    B_n^{E}(T_n^{j}(\omega)) \in SL(2,\R) \setminus SO(2,\R), \forall0 \leq j \leq r-1 
    \end{eqnarray}
    we denote 
    \begin{eqnarray*}
         u_j(E) = u(B^{E}_n(T_n^{j}(\omega))), s_j(E) = s(B^{E}_n(T_n^{j}(\omega))), \lambda_j(E) = \lambda(B^{E}_n(T_n^{j}(\omega)))
    \end{eqnarray*}
    for all $0 \leq j \leq r-1$.
  
    Denote
    \begin{eqnarray*}
       C^{E,j} &=& C^{E}(T_n^{j}(\omega)) \\
       B^{E} &=& A^{E}_n(T_n^{r-1}(\omega)) \cdots A^{E}_n(T_n(\omega))B_n^{E}(\omega) \\
       &=&B^{E}_n(T_n^{r-1}(\omega)) C^{E, r-1} \cdots C^{E,2}B^{E}_n(T_n(\omega))C^{E,1}B^{E}_n(\omega)   \\
       D_j(E)&=&R_{\frac{\pi}{2}-s_{j+1}(E)}C^{E,j+1}R_{u_j(E)} 
    \end{eqnarray*}
    for all $0 \leq j \leq r-2$.
\end{nota}
    By (P2),(P5) and (P6), when $\omega \in \Delta_{n+1}$ and $r = r_n(\omega)$ we have
    \begin{eqnarray*}
    B^{E} = B_{n+1}^{E}(\omega)
    \end{eqnarray*}
    
The main goal of this section is the following lemma, which says that under certain transversality conditions, we can give good lower bound for the norm of $B^{E}$ when $r$ is not too large, and at the same time, keep track of its stable,unstable directions. Similar estimates can be found in \cite{Bj}, \cite{WZ}, \cite{Y}. We need a slightly more precise estimate. We should notice that we only require the $C^{0}$ norm control of the stable/unstable directions. In this aspect, our lemma is simpler than the ones in those papers mentioned above.

\begin{lemma} \label{lemma proceed induction}
        There exists a constant $\Lambda > 0$ and $C_2, P > 0$ such that the following is true. For any $n\geq 0$,  $r \geq 1$,  $E \in I_{E_0}$, any $\omega \in \Delta_{n}$ such that \eqref{the last label} holds. Let $B^{E},u_j(E), s_j(E), \lambda_j(E)$ for $0\leq j\leq r-1$, $ D_j(E)$ for $0 \leq j \leq r-2$ be defined in Notation \ref{notation 2}. Assume that we have
    \begin{eqnarray}
    \label{lambda l is big}\lambda_l(E) &>& \bar{\lambda}_n > \Lambda, \forall 0 \leq l \leq r-1 \\
    \label{lower bounds of angles in the induction}   \angle( D_l(E)\begin{bmatrix}1\\0\end{bmatrix}, \begin{bmatrix}0\\1\end{bmatrix}) &>& \kappa_n >  2\bar{\lambda}_n^{-\frac{1}{4}}, \forall 0 \leq l \leq r-2 \\
    \label{r is small}r &<& C_2^{-1}\kappa_n^{3}\bar{\lambda}_n^{2}
    \end{eqnarray}

    Then $B^{E} \in SL(2,\R) \setminus SO(2,\R)$.
    Moreover,
    \begin{eqnarray*}
      \label{lower bounds of norms in the induction}   \lambda(B^E) &\geq&  C^{-Pr}\prod_{l=0}^{r-1} \lambda_{l}(E)\kappa_n^{r} \\
     \label{C1 bound of u n s n in the induction}   |s(B^E)-s_0(E)|_{\R/\pi\Z}, |u(B^E) - u_{r-1}(E)|_{\R/\pi\Z} &\leq& C^{P}\bar{\lambda}_n^{-2}\kappa_n^{-2}(r-1)
    \end{eqnarray*}
\end{lemma}

The key ingredient in the proof of Lemma \ref{lemma proceed induction} is the following lemma, which corresponds to the statement in Lemma \ref{lemma proceed induction} when $r=2$.

\begin{lemma} \label{main lemma}
      For any $C_0 \geq 1$, there exists $\hat{\lambda} > C_0$, such that
      for any $\bar{\lambda} > \hat{\lambda}$, $\kappa \in (0,1)$,
          for $\lambda_0, \lambda_1 > 0$, $D \in SL(2,\R)$, let
     \begin{eqnarray} 
        A &=& \begin{bmatrix} \lambda_1 & \\ & \lambda_1^{-1} \end{bmatrix} D\begin{bmatrix} \lambda_0 & \\ & \lambda_0^{-1} \end{bmatrix}  
      \end{eqnarray} 
      If we have
      \begin{eqnarray}
      \label{C1 of D in main lemma}   \norm{D} &\leq& C_0 \\
      \label{lower bound of norms in main lemma}  \min(\lambda_0, \lambda_1) &\geq& \bar{\lambda} \\
      \label{lower bound of angles in main lemma}   \angle( D\begin{bmatrix} 1 \\ 0 \end{bmatrix} , \begin{bmatrix} 0 \\ 1 \end{bmatrix} ) &>& \kappa > \bar{\lambda}^{-\frac{1}{4}}
      \end{eqnarray}
      then we have $A \in SL(2,\R) \setminus SO(2,\R) $.
      Moreover,
      \begin{eqnarray}
      \label{lower bound of final norm in main lemma}   \lambda(A)  &\gtrsim&    C_0^{-1}\lambda_0\lambda_1\kappa \\
      \label{C1 of u n s n in main lemma}   |\frac{\pi}{2} - s(A)|_{\R/\pi\Z}, |u(A)|_{\R/\pi\Z} &\lesssim& C_0^{O(1)}\bar{\lambda}^{-2}\kappa^{-2}
      \end{eqnarray}
\end{lemma}

\begin{proof}[Proof of Lemma \ref{main lemma}]
    We denote $$f(x) = \norm{A\begin{bmatrix}\cos x \\ \sin x \end{bmatrix}}^2 $$
     and $$D = \begin{bmatrix} a & b \\ c & d \end{bmatrix}$$
     By \eqref{C1 of D in main lemma}, we have that
     \begin{eqnarray}
     \label{C1 of a b c d} |a|,|b|,|c|,|d| \leq C_0
     \end{eqnarray}
     By \eqref{lower bound of angles in main lemma}, \eqref{C1 of D in main lemma} we have
     \begin{eqnarray}
     \label{lower bound of a}  |a| &=&|\langle D\begin{bmatrix} 1 \\ 0 \end{bmatrix}, \begin{bmatrix} 1\\0 \end{bmatrix}\rangle| \\
     &\geq& \frac{1}{\norm{D}}\sin\angle(D\begin{bmatrix} 1 \\ 0 \end{bmatrix} , \begin{bmatrix} 0 \\ 1 \end{bmatrix} ) > \frac{1}{10}C_0^{-1}\kappa  \nonumber
     \end{eqnarray}  
     Simple calculations show that
     \begin{eqnarray}
  \label{exp A}   A &=& \begin{bmatrix} \lambda_0\lambda_1 a & \lambda_0^{-1}\lambda_1 b \\  \lambda_1^{-1}\lambda_0c & \lambda_0^{-1}\lambda_1^{-1}d \end{bmatrix} 
     \end{eqnarray}
     Then
     \begin{eqnarray*}
     A\begin{bmatrix} \cos\alpha \\ \sin\alpha \end{bmatrix} &=& \begin{bmatrix}  \lambda_0\lambda_1 a\cos\alpha +  \lambda_0^{-1}\lambda_1 b \sin\alpha \\ \lambda_1^{-1}\lambda_0c \cos\alpha + \lambda_0^{-1}\lambda_1^{-1}d  \sin\alpha \end{bmatrix}
     \end{eqnarray*}
     and
     \begin{eqnarray}
    \label{f express} f(\alpha) = ( \lambda_0\lambda_1 a \cos\alpha +  \lambda_0^{-1}\lambda_1 b \sin\alpha)^2 + \\ (\lambda_1^{-1}\lambda_0c \cos\alpha + \lambda_0^{-1}\lambda_1^{-1}d \sin\alpha)^2  \nonumber
     \end{eqnarray}
    Note that by  \eqref{exp A}, \eqref{C1 of D in main lemma}, \eqref{lower bound of norms in main lemma}, \eqref{lower bound of a} and the second inequality in \eqref{lower bound of angles in main lemma}, we have
    \begin{eqnarray}
    \label{trace} tr(A) & =& \lambda_0 \lambda_1 a + \lambda_0^{-1}\lambda_1^{-1}d  \\
    &\geq& \frac{1}{10C_0}\bar{\lambda}^{2} \kappa - \bar{\lambda}^{-2}C_0  \nonumber \\
    &\geq& \frac{1}{10C_0}\hat{\lambda} - \hat{\lambda}^{-2}C_0  \nonumber \\
    &>& 2 \nonumber
    \end{eqnarray}
    when $\hat{\lambda}$ is bigger than some constant depending only on $C_0$. Thus $A \in SL(2,\R) \setminus SO(2,\R)$ when $\hat{\lambda}$ is sufficiently large depending only on $C_0$.
     
    Define $\theta$ by setting
    \begin{eqnarray*}
    f(\theta) &=& \sup_{\alpha} f(\alpha)
    \end{eqnarray*}
    Note that by \eqref{trace}, $\theta$ is uniquely defined up to a multiple of $\pi$.
    
    Then 
    \begin{eqnarray}
    \label{eqn for theta}  f'(\theta) &=& 0
    \end{eqnarray}
    We have
    \begin{eqnarray}
    \label{partial x f}  f'( \alpha) &=& 2 \cos(2\alpha) L_1 + \sin(2\alpha) L_2 
    \end{eqnarray}
    where
    \begin{eqnarray}
    \label{eqn for L 1}   L_1 &=& \lambda_1^2 ab+ \lambda_1^{-2}cd   \\
    \label{eqn for L 2}   L_2 &=& -\lambda_0^2\lambda_1^2 a^2 + \lambda_0^{-2}\lambda_1^2b^2 - \lambda_0^2\lambda_1^{-2}c^2 + \lambda_0^{-2}\lambda_1^{-2}d^2
    \end{eqnarray}
    By \eqref{eqn for theta},\eqref{partial x f} we have either
    \begin{eqnarray*}
    \theta &=& \frac{1}{2}\tan^{-1}(-\frac{2L_1}{L_2} ) \mod \pi\Z  \\
    \mbox{  or \hspace{1 cm} }  \theta &=& \frac{1}{2}\tan^{-1}(-\frac{2L_1}{L_2} ) + \frac{\pi}{2} \mod \pi\Z
    \end{eqnarray*}
    here we consider function $\tan$ as a function from $(-\frac{\pi}{2}, \frac{\pi}{2})$ to $\R$.
    
        Now we estimate $|L_1|,|L_2|$.
    
    By (\ref{C1 of a b c d}) and (\ref{eqn for L 1}), we have
    \begin{eqnarray}
    \label{L 1}   |L_1| &<& 2\lambda_1^2 C_0^2 
    \end{eqnarray}

    When $\bar{\lambda}$ is sufficiently large depending only on $C_0$,
    by \eqref{lower bound of a},\eqref{lower bound of angles in main lemma} and \eqref{eqn for L 2} we have
    \begin{eqnarray}
    \label{L 2}   C_0^2\lambda_0^2\lambda_1^2    \gtrsim  |L_2| \gtrsim C_0^{-2}\lambda_0^2\lambda_1^2 \kappa^2 
    \end{eqnarray}

    Hence by \eqref{L 1} and \eqref{L 2}
    \begin{eqnarray}
   \label{tan small}  |\frac{1}{2}\tan^{-1}(-\frac{2L_1}{L_2} )| \lesssim |\frac{L_1}{L_2}| \lesssim C_0^4\lambda_0^{-2}\kappa^{-2}
    \end{eqnarray}
    
    Now we are going to compare $f( \frac{1}{2}\tan^{-1}(-\frac{2L_1}{L_2} ))$ and $f( \frac{1}{2}\tan^{-1}(-\frac{2L_1}{L_2} ) + \frac{\pi}{2})$.
    
    If $\theta = \frac{1}{2}\tan^{-1}(-\frac{2L_1}{L_2} )$, when $\hat{\lambda}$ is bigger than some constant depending only on $C_0$, by \eqref{f express}, \eqref{partial x f}, \eqref{lower bound of a}, \eqref{lower bound of norms in main lemma} ,\eqref{C1 of a b c d} and the second inequality in \eqref{lower bound of angles in main lemma} we have that
    \begin{eqnarray*}
    f(\theta) &\gtrsim& (\lambda_0 \lambda_1 C_0^{-O(1)} \kappa- \lambda_0^{-1}\lambda_1C_0)^{2} -C_0^2 (\lambda_1^{-1}\lambda_0 + \lambda_0^{-1}\lambda_1^{-1})^2  \\
    &\gtrsim& C^{-O(1)}\lambda_0^{2}\lambda_1^{2}\kappa^{2} \gtrsim C^{-O(1)}\max(\lambda_0^{2},\lambda_1^{2})\bar{\lambda}^{\frac{3}{2}}
    \end{eqnarray*}

    If $\theta = \frac{1}{2}\tan^{-1}(-\frac{2L_1}{L_2} ) + \frac{\pi}{2}$, when $\hat{\lambda} > 1$, by \eqref{f express},\eqref{partial x f},  \eqref{lower bound of norms in main lemma}, \eqref{C1 of a b c d}, \eqref{tan small} and the second inequality in \eqref{lower bound of angles in main lemma} we have that
    \begin{eqnarray*}
    f(\theta) &\lesssim& (\lambda_0^{-1}\lambda_1C_0^{O(1)}\kappa^{-2} + \lambda_0^{-1}\lambda_1 C_0)^{2} + C_0^{2}(\lambda_1^{-1}\lambda_0 + \lambda_0^{-1}\lambda_1^{-1})^2 \\
    &\lesssim& C^{O(1)}(\lambda_1^{2} + \lambda_0^{2})\kappa^{-4} \lesssim C^{O(1)}\max(\lambda_1^{2} , \lambda_0^{2})\bar{\lambda}
    \end{eqnarray*}
    
        Thus we have showed that $f(\frac{1}{2}\tan^{-1}(-\frac{2L_1}{L_2} )) > f(\frac{1}{2}\tan^{-1}(-\frac{2L_1}{L_2} ) + \frac{\pi}{2})$ when $\bar{\lambda}$ is sufficiently large depending only on $C_0$. Since clearly that $f$ is $\pi-$periodic, this implies that we can take
        \begin{eqnarray} \label{theta equal blahblah}
             \theta =  \frac{1}{2}\tan^{-1}(-\frac{2L_1}{L_2} )
        \end{eqnarray}
    
    Since $\bar{\lambda} > \hat{\lambda}$, when $\hat{\lambda}$ is sufficiently large depending only on $C_0$, 
    by \eqref{theta equal blahblah} and \eqref{tan small} we have 
    \begin{eqnarray*}
      \label{bound for theta}   |\theta| &\lesssim& C_0^4\lambda_0^{-2}\kappa^{-2}
    \end{eqnarray*}

      By definition, we have $\theta = s- \frac{\pi}{2}$ modulo $\pi\Z$.
    Thus we have
    \begin{eqnarray*}
    |\frac{\pi}{2}-s|_{\R/\pi\Z} &\lesssim& C_0^{4}\lambda_0^{-2}\kappa^{-2} < C_0^{4}\bar{\lambda}^{-2}\kappa^{-2} 
    \end{eqnarray*}
    By symmetry, we have
    \begin{eqnarray*}
    |u|_{\R/\pi\Z} &\lesssim& C_0^{4}\bar{\lambda}^{-2}\kappa^{-2}
    \end{eqnarray*}
    This proves (\ref{C1 of u n s n in main lemma}).
   
    By  (\ref{C1 of D in main lemma}), (\ref{lower bound of norms in main lemma}) and (\ref{lower bound of angles in main lemma}),
    \begin{eqnarray*}
    f &\gtrsim& C_0^{-2}\lambda_0^2\lambda_1^2\kappa^2
    \end{eqnarray*}
    It is easy to see that
    $$\sigma = f(\theta)^{\frac{1}{2}}$$
    Hence
    \begin{eqnarray*}
        \sigma &\gtrsim& C_0^{-1}\lambda_0\lambda_1\kappa
    \end{eqnarray*}
    This proves (\ref{lower bound of final norm in main lemma}).
    
\end{proof}

\begin{proof}[Proof of Lemma \ref{lemma proceed induction}]
Recall that $B^{E}, u_k(E),s_k(E),\lambda_k(E)$ are defined in Notation \ref{notation 2}. By definition
\begin{eqnarray*}
B_n^{E}(T_n^{j}(\omega)) &=& R_{u_j(E)}\begin{bmatrix} \lambda_j(E) & 0 \\ 0 & \lambda_j(E)^{-1} \end{bmatrix} R_{\frac{\pi}{2} - s_j(E)}, \forall 0 \leq j \leq r-1 
\end{eqnarray*}
Then
\begin{eqnarray*}
   B^{E} = R_{u_{r-1}(E)}\begin{bmatrix}\lambda_{r-1}(E)&\\&\lambda_{r-1}(E)^{-1}\end{bmatrix} D_{r-2} \begin{bmatrix}\lambda_{r-1}(E)&\\&\lambda_{r-1}(E)^{-1}\end{bmatrix} \cdots \\ D_1 \begin{bmatrix}\lambda_1(E)&\\&\lambda_1(E)^{-1}\end{bmatrix} D_0 \begin{bmatrix}\lambda_0(E)&\\&\lambda_0(E)^{-1}\end{bmatrix} R_{\frac{\pi}{2} - s_0(E)}
\end{eqnarray*}
For all $l\geq 0$, we denote
\begin{eqnarray*}
   B^{(l),E} =\begin{bmatrix}\lambda_l(E)&\\&\lambda_l(E)^{-1}\end{bmatrix} D_{l-1} \begin{bmatrix}\lambda_{l-1}(E)&\\&\lambda_{l-1}(E)^{-1}\end{bmatrix} \cdots \\ D_1 \begin{bmatrix}\lambda_1(E)&\\&\lambda_1(E)^{-1}\end{bmatrix} D_0 \begin{bmatrix}\lambda_0(E)&\\&\lambda_0(E)^{-1}\end{bmatrix} 
\end{eqnarray*}
In particular, we have
\begin{eqnarray*}
   B^{(0),E}&=& \begin{bmatrix}\lambda_0(E)&\\&\lambda_0(E)^{-1}\end{bmatrix}
\end{eqnarray*}
For any $l$ such that $B^{(l),E} \in SL(2,\R) \setminus SO(2,\R)$, we denote functions 
\begin{eqnarray*}
u_{(l)} = u(B^{(l),E}), \hspace{0.5cm}s_{(l)} = s(B^{(l),E}),\hspace{0.5cm}\sigma_{l} = \lambda(B^{(l),E})
\end{eqnarray*}
where
\begin{eqnarray}
\label{us0} u_{(0)} = 0 ,\hspace{0.5cm}
s_{(0)} = \frac{\pi}{2} ,\hspace{0.5cm}
\sigma_0 = \lambda_0 
\end{eqnarray}
We have
\begin{eqnarray}
\label{tmp 3 1} B^{(l+1),E} &=& \begin{bmatrix} \lambda_{l+1}(E) & \\ & \lambda_{l+1}(E)^{-1} \end{bmatrix} D_{l}(E) 
 R_{u_{(l)}(E)} \\ \cdot \begin{bmatrix} \sigma_{l}(E) & \\ & \sigma_{l}(E)^{-1} \end{bmatrix}  R_{\frac{\pi}{2} - s_{(l)}(E)}  \nonumber\\
\label{be} B^{E} &=& R_{u_{r-1}(E)} B^{(r-1),E} R_{\frac{\pi}{2} - s_0(E)}  
\end{eqnarray}

We will inductively show that for some absolute constant $P > 0$, for all $0 \leq l \leq r-1$ we have,
\begin{eqnarray}
\label{hyp ind 2} \sigma_l &>& \bar{\lambda}_n   \\
\label{hyp ind 3} |u_{(l)}|_{\R/\pi\Z}, |\frac{\pi}{2} - s_{(l)}|_{\R/\pi\Z} &\leq& C^{P} l\kappa_n^{-2}  \bar{\lambda}_n^{-2}
\end{eqnarray}

By \eqref{us0}, we clearly have \eqref{hyp ind 2}, \eqref{hyp ind 3} for $l=0$.

Assume that for some $0\leq l \leq r-2$,  \eqref{hyp ind 2} and \eqref{hyp ind 3} are valid and $B^{(l),E} \in SL(2,\R) \setminus SO(2,\R)$.
By \eqref{hyp ind 3} for $l$, we apply Lemma \ref{C0 close then angle close}, \eqref{r is small} and \eqref{lower bounds of angles in the induction} to see that
\begin{eqnarray}
\angle(D_l (E)R_{u_{(l)}}\begin{bmatrix} 1 \\ 0 \end{bmatrix} ,  \begin{bmatrix} 0 \\ 1 \end{bmatrix}  ) &>& \angle(D_l (E) \begin{bmatrix} 1 \\ 0 \end{bmatrix} ,  \begin{bmatrix} 0 \\ 1 \end{bmatrix}  ) - c_5 C^{O(1)} l\kappa_n^{-2}  \bar{\lambda}_n^{-2} \nonumber \\
&>& \kappa_n -c_5C^{O(1)} r\kappa_n^{-2}  \bar{\lambda}_n^{-2}  \nonumber \\
\label{tmp condition 3} &>& \frac{1}{2}\kappa_n 
\end{eqnarray}
The last inequality is true by \eqref{r is small} when we choose $C_2$ to be sufficiently large. We note that $C_2$ can be taken to be an absolute constant.

By \eqref{lambda l is big}, we have
\begin{eqnarray*}
\lambda_{l+1}(E) &>& \bar{\lambda}_n
\end{eqnarray*}
When $\Lambda > \hat{\lambda}$ where $\hat{\lambda}$ is given by Lemma \ref{main lemma} with $C_0 = C$, we apply Lemma \ref{main lemma} for 
$
\lambda_1=\lambda_{l+1}(E), 
\lambda_0=\sigma_{l}(E),
D= D_{l}(E)R_{u_{(l)}}
$.
We note that
by \eqref{hyp ind 2} for $l$, \eqref{tmp condition 3}, \eqref{lower bounds of angles in the induction} and \eqref{lambda l is big} that the condition of Lemma \ref{main lemma} is satisfied for $\kappa = \frac{1}{2}\kappa_n$.

By Lemma \ref{main lemma}, we have
\begin{eqnarray*} \begin{bmatrix} \lambda_{l+1}(E) & \\ & \lambda_{l+1}(E)^{-1} \end{bmatrix} D \begin{bmatrix} \sigma_{l}(E) & \\ & \sigma_{l}(E)^{-1} \end{bmatrix} \in SL(2,\R) \setminus SO(2,\R)
\end{eqnarray*}
By \eqref{tmp 3 1}, we obtain that $B^{(l+1),E} \in SL(2,\R) \setminus SO(2,\R)$,
Moreover,
\begin{eqnarray*} &&\begin{bmatrix} \lambda_{l+1}(E) & \\ & \lambda_{l+1}(E)^{-1} \end{bmatrix} D_{l}(E) 
 R_{u_{(l)}(E)} \begin{bmatrix} \sigma_{l}(E) & \\ & \sigma_{l}(E)^{-1} \end{bmatrix} \\
 &=& R_{u_{(l+1)}}\begin{bmatrix} \sigma_{l+1}(E) & \\ & \sigma_{l+1}(E)^{-1} \end{bmatrix} R_{s_{(l)}-s_{(l+1)}}
\end{eqnarray*}
Then by Lemma \ref{main lemma} and the fact we assumed $C > 1$, we see that by enlarging $P$ if necessary, we obtain
\begin{eqnarray}
\label{tmp 3 2}|u_{(l+1)}|_{\R/\pi\Z}, |s_{(l)}-s_{(l+1)}|_{\R/\pi\Z} &\leq& C^{P}\bar{\lambda}_n^{-2}\kappa_n^{-2} \\
\label{tmp 3 3}\sigma_{l+1} &\geq& C^{-P}\lambda_{l+1} \sigma_{l} \kappa_n
\end{eqnarray}
We note that we can choose $P$ to be an absolute constant.

By \eqref{hyp ind 3}, \eqref{tmp 3 2}, we have
\begin{eqnarray*}
|u_{(l+1)}|_{\R/\pi\Z}, |\frac{\pi}{2} - s_{(l+1)}|_{\R/\pi\Z} &\leq& C^{P} (l+1)\kappa_n^{-2}  \bar{\lambda}_n^{-2}
\end{eqnarray*}
This recoved estimate \eqref{hyp ind 3} for $l+1$.

Since by \eqref{hyp ind 2} and the second inequality in \eqref{lower bounds of angles in the induction}, we see that
\begin{eqnarray*}
\sigma_{l} \kappa_n > \bar{\lambda}_n^{\frac{1}{2}} \geq \Lambda^{\frac{1}{2}}
\end{eqnarray*}
Then by \eqref{tmp 3 3} and  \eqref{lambda l is big} we have
\begin{eqnarray*}
\sigma_{l+1} &>& \bar{\lambda}_n
\end{eqnarray*}
when $\Lambda$ is sufficiently large depending only on $C$.
Hence we have recovered estimates \eqref{hyp ind 2} for $l+1$ and have completed the induction.
Moreover, we see that \eqref{tmp 3 3} holds for any $0 \leq l \leq r-2$.

By \eqref{hyp ind 2} for $l=r-1$, we get
\begin{eqnarray}
\label{usr} |u_{(r-1)}(E)|_{\R/\pi\Z}&\leq& C^{P}\bar{\lambda}_n^{-2}\kappa_n^{-2}(r-1) \\
\label{usr2} |\frac{\pi}{2} - s_{(r-1)}(E)|_{\R/\pi\Z} &\leq& C^{P}\bar{\lambda}_n^{-2}\kappa_n^{-2}(r-1)
\end{eqnarray}

Concatenating the estimates \eqref{tmp 3 3} for $0\leq l \leq r-2$, and using $C > 1$, we get
\begin{eqnarray}
\label{sigmar-1} \sigma_{r-1}(E) > C^{Pr} \kappa_n^{r-1} \prod_{i=0}^{r-1}\lambda_{i}(E)
\end{eqnarray}

By \eqref{be}, we have that
\begin{eqnarray*}
\lambda(B^E) &=& \sigma_{r-1}   \\
u(B^E) &=& u_{r-1} + u_{(r-1)}  \\
s(B^E) &=& s_0 + s_{(r-1)} -\frac{\pi}{2}
\end{eqnarray*}
Then the lemma follows from \eqref{usr}, \eqref{usr2} and \eqref{sigmar-1}
\end{proof}


\section{Choosing the parameters} \label{choosing the parameters}

In this section, we will introduce several sets that will help us estimate the area of the spectrum in Section \ref{cover the spectrum}, \ref{Area of the spectrum and the proof of the main theorem}.

\begin{definition} \label{nota J}
For any $n \geq 0$, we define
\begin{eqnarray*}
\mathcal{A}_n = \bigcup_{i=1}^{K}\{ \omega_i\omega_{i+1}\cdots \omega_{l_n(\omega)-1} ; \omega \in \Delta_n \bigcap \Delta_{(i)} \}
\end{eqnarray*}

For $\alpha, \beta \in \mathcal{A}_n, 1\leq j \leq K$, $\epsilon > 0$, we define $J(\alpha,\beta,j,\epsilon) \subset I_{E_0}$ as
\begin{eqnarray}
J(\alpha,\beta,j, \epsilon) &=& \{ E \in I_{E_0} ; A^{E}(\alpha), A^{E}(\beta) \in SL(2,\R) \setminus SO(2,\R) \mbox{ and }\nonumber  \\ &&\angle(  R_{\frac{\pi}{2}-s(A^E(\beta))}  C^E_j R_{u (A^E(\alpha))} \begin{bmatrix} 1 \\ 0 \end{bmatrix},\begin{bmatrix} 0 \\ 1 \end{bmatrix}) \leq  \epsilon  \}   \nonumber  
\end{eqnarray}
For a given choice of the sequence $\{\Delta_n\}_{n \in \N}$ ( which in turn determines $\{A^{E}_n\}_{n \in \N}$,$\{B^{E}_n\}_{n \in \N}$, etc.) and $\{\kappa_n\}_{n \in \N}$, we define
\begin{eqnarray}
\label{def of J n}J_n  &=& \bigcup_{\alpha\in \mathcal{A}_n, \beta \in \mathcal{A}_n, 1 \leq j \leq K } J(\alpha, \beta, j, \kappa_n)   \\
\label{def of J}J   &=& \bigcup_{n} J_n 
\end{eqnarray}
\end{definition}

By Definition \ref{nota J}, (P5) we see that, for any $n\geq 0$, $1\leq j\leq K$, any $\omega, \tilde{\omega} \in \Delta_n$, any $E \in J_n$, we have
\begin{eqnarray}
\label{jn1}B^{E}_n(\omega), B^{E}_n(\tilde{\omega}) \in SL(2,\R) &\setminus& SO(2,\R) \\
\label{jn2}\mbox{ and  } \angle(R_{\frac{\pi}{2} - s(B^{E}_n(\omega))}C^{E}_j R_{u(B^{E}_n(\tilde{\omega}))} \begin{bmatrix} 1 \\ 0 \end{bmatrix},\begin{bmatrix} 0 \\ 1 \end{bmatrix}) &\leq& \kappa_n 
\end{eqnarray}

Now we will choose the parameters $\bar{\lambda}_n, \zeta_n, \chi_n, M_n, N_n, \kappa_n$ which were introduced in Section \ref{a tower construction}. In the rest of this paper, we use the following notation.

\begin{nota}
For any $n \geq 0$, we denote 
\begin{eqnarray*}
\inf l_n &=& \inf_{\omega \in \Delta_{n}} l_{n}(\omega)  \\
\sup l_n &=& \sup_{\omega \in \Delta_{n}} l_{n}(\omega)
\end{eqnarray*}
\end{nota}

The goal of this section is to show the following lemma.

\begin{lemma} \label{lemma choosing the parameters}
For any $0 < \gamma < \gamma' < \frac{1}{4}$, any $0 < c < 2-3\gamma' $, there exists $C', C", \Gamma > 0$ such that the following is true. For any admissible potential $v$, denote $\lambda = \lambda_{v}$, such that $ \lambda > \Gamma$, then there exists $\{\Delta_n\}_{n \in \N}$, and parameters $ \bar{\lambda}_n, \zeta_n, \chi_n, M_n, N_n, \kappa_n$ such that :

Let $\bar{\lambda}_0, \chi_0, M_0$ be given by Lemma \ref{lemma initial data}.
For any $n\geq 0$, we define $J_n$ by \eqref{def of J n}. Then we have \eqref{ratios of time intervals} to \eqref{lower bound of norm}.
Moreover, for all $n \geq 0$ we have
\begin{eqnarray}
\label{item 1}\zeta_n < \bar{\lambda}_n^{-c} \leq\bar{\lambda}_0^{-2^{n}c} \\
\label{item 2} \chi_{n} > C'\chi_0\\
\label{item 3} \bar{\lambda}_n^{-\gamma'} < \kappa_n < \bar{\lambda}_0^{-\gamma} \\
\label{item 4} M_n \leq C" M_0
\end{eqnarray}
\end{lemma}
\begin{proof}
Denote \begin{eqnarray} \label{tmp 0} \xi = -\frac{1}{10}\log \gamma' \end{eqnarray}
By the condition $\gamma' < 1$, we have $\xi > 0$.

We choose an arbitrary sequence of integers $\{N_n\}$, such that 
\begin{eqnarray}
\label{tmp 1} \sum_{n=0}^{\infty} \log\frac{N_n+1}{N_n} &<& \xi  \\
\label{tmp 2}  \frac{2}{1 - e^{-\xi}} &\leq& N_n, \forall n \geq 0 \\
\label{tmp 100} N_{n+1} &\leq& 2N_{n}, \forall n \geq 0
\end{eqnarray}
By \eqref{tmp 1}, we get \eqref{N n inverse summable}.

Assume that $\Delta_m$ is defined for all $0 \leq m \leq n$ for some $n \geq 0$ ( $\Delta_0$ is defined in Section \ref{a tower construction} ). By Rokhlin tower theorem and aperiodicity, we can and do choose
\begin{eqnarray*}
\Delta_{n+1} \subset \Delta_{n}
\end{eqnarray*}
 such that for any $  \omega \in \Delta_{n+1}$, we have 
\begin{eqnarray*}
r_{n}(\omega) \in \{N_n, N_n+1\}
\end{eqnarray*}
We inductively define $\Delta_n$ for all $n \geq 0$ and we get \eqref{r n well controlled} for all $n \geq 0$.

We define that
\begin{eqnarray}
\label{tmp 3} M_{n+1} = \frac{N_n+1}{N_n}M_n, \forall n\geq 0
\end{eqnarray}
where $M_0$ is defined in Lemma \ref{lemma initial data} as $M_0 = \frac{\sup l_0}{\inf l_0}$. Since we already showed \eqref{r n well controlled}, by \eqref{used once no} and \eqref{r n well controlled} we have for any $n\geq 0$, for any $\omega \in \Delta_{n+1}$,
\begin{eqnarray*}
l_{n+1}(\omega) \leq r_n(\omega) \sup l_n \leq (N_{n}+1)\sup l_n
\end{eqnarray*}
and
\begin{eqnarray*}
l_{n+1}(\omega) \geq r_n(\omega) \inf l_n \geq N_n \inf l_n
\end{eqnarray*}
If we have $\sup l_n \leq M_n \inf l_n$, then we have
\begin{eqnarray*}
\sup l_{n+1} \leq \frac{N_n+1}{N_n}M_n\inf l_{n+1} = M_{n+1} \inf l_{n+1}
\end{eqnarray*}
This gives \eqref{ratios of time intervals} for all $n \geq 0$.

By \eqref{tmp 1} and \eqref{tmp 3}, we obtain \eqref{item 4} with $C" = e^{\xi}$.

We choose an arbitrary sequence $\{\eta_n\}_{n\in \N}$ that satisfy
\begin{eqnarray}
\label{tmp 20}\sum_{n=0}^{\infty} \eta_n  &<& \infty \\
\label{tmp 4}\mbox{ and  } \frac{\gamma}{2^{n}} < \eta_n \leq \eta_0 &<& \gamma', \forall n \in \N 
\end{eqnarray}

Let $P$ be given by Lemma \ref{lemma proceed induction}. We define for all $n \geq 0$
\begin{eqnarray}
\label{tmp 6}\kappa_n &=& \bar{\lambda}_n^{-\eta_n} \\
\label{tmp 5} \chi_{n+1} &=& \inf_{\omega \in \Delta_{n}} \inf_{1\leq r \leq r_n(\omega)}( \chi_{n} + \frac{r( \log\kappa_n-P \log C)}{\sum_{i=0}^{r-1} l_n(T_n^{i}(\omega))}) \\
\label{tmp 30}\bar{\lambda}_{n+1} &=& e^{\chi_{n+1} \inf l_{n+1}} 
\end{eqnarray}

Now we are going to verify \eqref{item 2} and the second inequality in \eqref{item 3}.
We first show the following lemma.
\begin{lemma}
There exists $C' > 0$ such that we have for all sufficiently large $\bar{\lambda}_0 > 0$ the following
\begin{eqnarray}
\label{tmp 800}\chi_{n+1} &> &C'\chi_0 \\
\label{tmp 8} \bar{\lambda}_n &\geq& \bar{\lambda}_0^{2^{n}}
\end{eqnarray}
for all $n \geq 0$. 

As a consequence, for all sufficiently large $\bar{\lambda}_0$ we have \eqref{item 2} and the second inequality in \eqref{item 3} for all $n \geq 0$.
\end{lemma}

\begin{proof}
By \eqref{tmp 30} and Lemma \ref{lemma initial data} we have
\begin{eqnarray}
\label{tmp 7}\log \bar{\lambda}_n \leq  \chi_n \inf l_n
\end{eqnarray}
for all $n \geq 0$.

Hence by \eqref{tmp 6}, \eqref{tmp 5} and \eqref{tmp 7}, we have for all $n \geq 0$,
\begin{eqnarray}
\label{chichichi}\chi_{n+1} &\geq& \chi_{n} + \frac{1}{\inf l_n}(-\eta_n \log\bar{\lambda}_n-P \log C ) \\
&\geq& \chi_n + \frac{1}{\inf l_n}( -\eta_n \inf l_n \chi_n- P \log C) \nonumber \\
&\geq& \chi_{n}(1-\eta_n) - P \log C \nonumber
\end{eqnarray}

Since by \eqref{tmp 0},\eqref{tmp 20} and \eqref{tmp 4}, we have
\begin{eqnarray*}
\eta_n &<&  \gamma' < 1, \forall n\geq 0 \\
\mbox{ and } \sum_{n=0}^{\infty} \eta_n &<& \infty
\end{eqnarray*}
Then there exists $C' > 0$, $C" > 0$ depending only on $\xi, P, C, M_0$, such that if $\chi_0 > C"$, then we have
\begin{eqnarray}
\chi_{n+1} > C'\chi_0, \forall n \geq 0
\end{eqnarray}
This proves \eqref{tmp 800} and \eqref{item 2}.

To simplify the notations, by \eqref{chichichi}, we note that we can choose $C"$ to be large, still depending only on $\xi, P, C, M_0$ such that : if $\chi_0 > C"$, then
\begin{eqnarray*} \chi_{n+1} \geq \chi_{n}(1 - e^{\xi} \gamma')  \end{eqnarray*}
for all $n \geq 0$.

Then it is clear by \eqref{tmp 30} that for all $n \geq 0$
\begin{eqnarray*}
\bar{\lambda}_{n+1} &=& e^{\chi_{n+1} \inf l_{n+1}}  
\geq e^{ \chi_n ( 1 - e^{\xi} \gamma') N_n \inf l_n} 
\geq \bar{\lambda}_n^{2}
\end{eqnarray*}
The last inequality follows from $N_n(1 - e^{\xi}\gamma') \geq 2$ by \eqref{tmp 0} and \eqref{tmp 2}.
This shows that we have 
\begin{eqnarray}
 \bar{\lambda}_n \geq \bar{\lambda}_0^{2^{n}}
\end{eqnarray}
This proves \eqref{tmp 8}.

By \eqref{tmp 4}, \eqref{tmp 6}, \eqref{tmp 8} we have
\begin{eqnarray*}
\kappa_n = \bar{\lambda}_n^{-\eta_n} 
\leq \bar{\lambda}_0^{-2^{n}\eta_n} 
< \bar{\lambda}_0^{-\gamma}
\end{eqnarray*}
This proves the second inequality in \eqref{item 3}.

\end{proof}

Now we will define $J_n$ inductively and verify \eqref{B E hyperbolic} to \eqref{lower bound of norm} along the way.

For $n=0$, we obtain \eqref{B E hyperbolic} and \eqref{lower bound of norm} by Lemma \ref{lemma initial data} when $\lambda$ is sufficiently large.

Assume that for $n \geq 0$, we have defined $J_0, \cdots, J_{n-1}$ and we have \eqref{B E hyperbolic} and \eqref{lower bound of norm} for $0$ to $n$. We define $J_n$ by \eqref{def of J n}. 

By \eqref{lower bound of norm} for $n$ and \eqref{tmp 4}, \eqref{tmp 6}, for any $E \in I_{E_0} \setminus \bigcup_{n \geq m \geq 0}J_m$, for any $\tilde{\omega} \in \Delta_n$, we have
\begin{eqnarray}
\label{tmp 600} \kappa_n &>& \bar{\lambda}_n^{-\gamma'} \\
\label{tmp 700}\lambda(B^{E}_n(\tilde{\omega})) &\geq& \bar{\lambda}_n
\end{eqnarray}
In particular, we see that the first inequality in \eqref{item 3} for $n$ is valid.

We define that
\begin{eqnarray}
\label{tmp 9}\zeta_{n} &=&  C^{P}\bar{\lambda}_n^{-2+2\eta_n}N_n  
\end{eqnarray}

By $c < 2-3\gamma'$ ,\eqref{tmp 6} and \eqref{tmp 8}, for $\lambda$ larger than some constant depending only on $C$, we have that
\begin{eqnarray}
\label{tmp 400}\kappa_n^{3}\bar{\lambda}_n^{2} = \bar{\lambda}_n^{2-3\eta_n}
\geq \bar{\lambda}_n^{2-3\gamma'} 
\geq \bar{\lambda}_0^{2^{n}c}
\end{eqnarray}
By \eqref{tmp 100}, we see that
\begin{eqnarray}
\label{tmp 300}N_n \leq 2^n N_0
\end{eqnarray}
Let $C_2 > 0$ be given by Lemma \ref{lemma proceed induction}.
Hence by \eqref{tmp 400} and \eqref{tmp 300}, when $\bar{\lambda}_0$ is sufficiently large, we have
\begin{eqnarray}
\label{tmp 500}\sup_{\omega \in \Delta_{n+1}} r_n(\omega) \leq N_n+1 \leq C_2^{-1} \kappa_n^{3}\bar{\lambda}_n^{2}
\end{eqnarray}

Combining \eqref{jn1}, \eqref{jn2}, \eqref{tmp 600}, \eqref{tmp 700}, \eqref{tmp 8} and \eqref{tmp 500}, we see that the condition of Lemma \ref{lemma proceed induction} is satisfied for any $\omega \in \Delta_{n}$, any $1\leq r\leq r_n(\omega)$ and any $E \in I_{E_0}\setminus \bigcup_{0 \leq m \leq n} J_m$ when $\bar{\lambda}_0$ is sufficiently large depending only on $C$.

Apply Lemma \ref{lemma proceed induction}, 
we get \eqref{C1 bound of u n}, \eqref{C1 bound of s n}, \eqref{lower bound of norm 2} for $n$ and \eqref{B E hyperbolic}, \eqref{lower bound of norm} for $n+1$ using \eqref{tmp 5} and \eqref{tmp 30}.
By induction, we see that \eqref{B E hyperbolic} to \eqref{lower bound of norm} and the first inequality in \eqref{item 3} are valid for all $n \geq 0$.

Finally by \eqref{tmp 9}, \eqref{tmp 300} we have that
\begin{eqnarray*}
\zeta_{n} &\leq& C^{P}\bar{\lambda}_n^{2\gamma'-2}2^{n}N_0 
\end{eqnarray*}
By $0< c < 2-3\gamma'$ and \eqref{tmp 8}, when $\bar{\lambda}_0$ is sufficiently large depending only on $C$, we have
\begin{eqnarray*}
\zeta_n < \bar{\lambda}_n^{-c} \leq \bar{\lambda}_0^{-2^{n}c} 
\end{eqnarray*}
This proves \eqref{item 1}.

By Lemma \ref{lemma initial data}, we see that $\bar{\lambda}_0$ tends to infinity as $\lambda$ tends to infinity. This concludes the proof.
\end{proof}


\section{Cover the spectrum} \label{cover the spectrum}

The goal of this section is to prove the following lemma, which shows that under suitable conditions the spectrum is covered by $\bar{J}$, where $J$ is introduce in Notation \ref{nota J}.
\begin{lemma}\label{lemma spectrum}
For any $0 < \gamma < \gamma' < \frac{1}{4}$, any $\gamma' < c < 2-3\gamma'$, for all sufficiently large $\lambda$, we define $J_n$ and parameters $ \bar{\lambda}_n, \zeta_n, \chi_n, M_n, N_n, \kappa_n$ that satisfy the conclusions of Lemma \ref{lemma choosing the parameters} with $\gamma, \gamma', c$. Then we have $\Sigma \bigcap I_{E_0}  \subset \overline{J}$. Here $J$ is defined by \eqref{def of J} in Section \ref{choosing the parameters}.\end{lemma}
\begin{proof}
By Lemma \ref{lemma choosing the parameters}, for all sufficiently large $\lambda$, we have \eqref{ratios of time intervals}  to \eqref{lower bound of norm} and \eqref{item 1} to \eqref{item 4} for all $n \geq 0$.

Let $\hat{\chi} = \log \sup_{E \in I_{E_0}, \alpha \in \mathcal{A}} \norm{A^E_{\alpha}} $.
By \eqref{item 2}, we see that there exists $C' > 0$ such that $\chi_n > C' \chi_0 > 0$ for all $n\geq 0$. Hence we can take $c_2 > 0$ be a constant so that 
\begin{eqnarray}
\label{tmp 2 4}\hat{\chi} < c_2 \chi_n, \forall n \geq 0
\end{eqnarray}
By the choice of $\chi_{n}$ in \eqref{lower bound of norm 2}, we also see that
\begin{eqnarray}
\label{hat chi chi n}\chi_{n} \leq \hat{\chi}, \forall n \geq 0
\end{eqnarray}


By ergodicity, for $\mu - a.e. \omega \in \Omega$ , we can and do define
\begin{eqnarray*}
t_1(n,  \omega) &=& \inf\{ k \geq 0;   T^{k}(\omega)\in \Delta_n \}   \\ 
t_2(n,  \omega) &=& \inf\{ k > 0:  T^{-k}(\omega) \in \Delta_n \}  
\end{eqnarray*}
such that $t_1(n, \omega),t_2(n, \omega) < \infty$.
Then for such $\omega \in \Omega$, we define
\begin{eqnarray*}
W_1(n, \omega) &=& T^{t_1(n,\omega)}(\omega)    \\
W_2(n, \omega) &=& T^{-t_2(n,\omega)}(\omega)
\end{eqnarray*}

It is direct to see that
\begin{eqnarray}
\label{W1W2}W_1(n,\omega) &=& T_n(W_2(n,\omega))   \mbox{  for $\mu-a.e. \omega \in \Omega$} \\
\label{tmp 2 2} t_1(n, \omega) &\leq& \sup l_n
\end{eqnarray}

Since $(\Omega, T, \mu)$ is ergodic, it is a standard fact that $(\Delta_n, T_n,\mu_n)$ is also ergodic.
Hence for $\mu_n -a.e. \omega \in \Delta_n$, we can and do define
\begin{eqnarray*}
s_1(n,\omega) &=& \inf\{ k \geq 0;   T_n^{k}(\omega)\in \Delta_{n+1} \}  \\
s_2(n,\omega) &=& \inf\{ k \geq 0;   T_n^{-k}(\omega)\in \Delta_{n+1} \} 
\end{eqnarray*}
such that $s_1(n, \omega), s_2(n, \omega) < \infty$. Then for such $\omega \in \Omega$, we define
\begin{eqnarray*}
W_3(n, \omega) &=& T_n^{s_1(n,\omega)}(\omega)    \\
W_4(n, \omega) &=& T_n^{-s_2(n,\omega)}(\omega)
\end{eqnarray*}

We define
\begin{eqnarray*}
\Omega_n = \{ \omega \in \Omega ; s_2(n, \omega_2(n, \omega) ), s_1(n, \omega_1(n, \omega)) > 2N_n^{\frac{1}{2}}   \}
\end{eqnarray*}
By the definition of $s_1(\cdot, \cdot),s_2(\cdot,\cdot)$, for any $\omega \in \Delta_n$ such that $s_1(n, \omega) > 0$ or $s_2(n, \omega) > 0$, we have $\omega \in \Delta_n \setminus \Delta_{n+1}$. Hence for any $\omega \in \Omega_n$, we have that $\omega_1(n, \omega), \omega_2(n, \omega) \in \Delta_n \setminus \Delta_{n+1} $.

\begin{lemma} \label{lemma area estimate}
There exists a constant $c_4 > 0$ such that
$ \mu(\Omega_n) > c_4 $ for all sufficiently large $n$. 
\end{lemma}
\begin{proof}
By \eqref{r n well controlled}, for all $m \geq 0$, $\Delta_{m+1}, T_m(\Delta_{m+1}), \cdots, T_m(^{N_m-1}\Delta_{m+1})$ are mutually disjoint and belong to $\Delta_m$. Moreover it is easy to see that their union takes up a proportion of $\Delta_{m}$ no less than $\frac{N_n}{N_n+1}$.
We see from (P2) that $T_0^{k_0} T_1^{k_1} \cdots T_{n-1}^{k_{n-1}}T_{n}^{k_n}\Delta_{n+1}$ for $0 \leq k_0 \leq N_0-1, 0 \leq k_1 \leq N_1-1, \cdots, 0 \leq k_{n-1} \leq N_{n-1} , 2N_{n}^{\frac{1}{2}} \leq k_n \leq N_n - 2N_n^{\frac{1}{2}}$ all belong to $\Omega_n \bigcap \Delta_0$, and are mutually disjoint for points in different sets have different landing time with respect to sequence $\Delta_1, \cdots, \Delta_{n+1}$. We know that $\mu(\Omega_n) \geq \mu(\Delta_0) \mu_0(\Omega_n \bigcap \Delta_0)> \mu(\Delta_0) \frac{N_n-4N_n^{\frac{1}{2}}}{N_n+1} \frac{N_{n-1}}{N_{n-1}+1} \cdots \frac{N_0}{N_0+1}$. This proves the lemma since we have chosen $N_n$ so that $\prod_{n=0}^{\infty} \frac{N_n}{N_n+1} > 0$.
\end{proof}

Define
\begin{eqnarray*}
\Omega' = \bigcap_{n = 1}^{ \infty} \bigcup_{m =n}^{\infty} \Omega_m
\end{eqnarray*}
By Lemma \ref{lemma area estimate}, we have
\begin{eqnarray*}
\mu(\Omega') > 0
\end{eqnarray*}

In order to prove Lemma \ref{lemma spectrum}, it suffices to prove that :

For all $\omega \in \Omega'$, we have
\begin{eqnarray*}
\Sigma_{\omega} \bigcap I_{E_0} \subset \overline{J}
\end{eqnarray*}

Assume the contrary, then there exists $\omega \in \Omega'$ such that $$\nu_{\omega}(I_{E_0} \setminus \overline{J}) > 0 $$
Here $\nu_{\omega}$ is the spectral measure of the Schr\"odinger operator associated to $\omega$.

Then by Theorem \ref{Berezansky}, there exists $E \in I_{E_0} \setminus \overline{J}$ such that the Schr\"odinger operator with potential $\{ v(T^n(\omega)_0) \}_{n \in \Z}$ admitting a generalized eigenfunction with polynomial growth ( in fact, of degree 1 since we are considering a one-dimensional operator, but this point is not essentially used as we can see from the proof ).
Thus there exists $h \in \R^2, c_3 > 0$, such that
\begin{eqnarray}
\label{poly growth}  \norm{A^{E}(m, \omega) h} \leq c_3 (1 + |m|) \norm{h},   \forall m \in \Z
\end{eqnarray}

By the definition of $\Omega'$, there exists arbitrarily large $n$, such that $\omega \in \Omega_{n}$.
Denote 
\begin{eqnarray*}
\omega^{1} &=& W_1(n, \omega), \hspace{1 cm}\omega^2 = W_2(n, \omega) \\
\omega^{3} &=& W_3(n, \omega^1),\hspace{1 cm}\omega^4 = W_4(n, \omega^2)
\end{eqnarray*}
and
\begin{eqnarray*}
t_1 &=& t_1(n, \omega), \hspace{0.3cm} s_1 = s_1(n, \omega^1), \hspace{0.3cm} s_2 = s_2(n, \omega^2)
\end{eqnarray*}
We verify by definition that
\begin{eqnarray}
\omega^3 &=& T_n^{s_1}(\omega^1) 
\end{eqnarray}
and
\begin{eqnarray*}
T_n(\omega^2) = \omega^1
\end{eqnarray*}
We also denote
\begin{eqnarray*}
 g &=& A^{E}(t_1, \omega)h
\end{eqnarray*}
Denote the argument of $g$ by $\theta(g)$. More precisely we have
\begin{eqnarray*}
g  =\norm{g} \begin{bmatrix}\cos\theta(g) \\ \sin\theta(g)\end{bmatrix}
\end{eqnarray*}
By \eqref{tmp 2 2} we have estimate
\begin{eqnarray}
\norm{g} &\geq& \norm{A^E(t_1, \omega)}^{-1}\norm{h} \\
&\geq&  e^{-t_1 \hat{\chi} }\norm{h} \nonumber \\
&\geq& e^{-\sup l_n \hat{\chi} } \norm{h} \nonumber
\end{eqnarray}

Since $\omega \in \Omega_n$, we have $s_1,s_2 > 0$. By definition
\begin{eqnarray*}
\omega^2 = T_n^{s_2}(\omega^4)
\end{eqnarray*}

There exist $1 \leq i_1,i_2 \leq K$, such that $$\omega^1 \in \Delta_{(i_1)}, \hspace{1 cm} \omega^4 \in \Delta_{(i_2)}$$
We denote that
\begin{eqnarray}
\label{l1}L_1 &=& \sum_{i=0}^{s_1-1} l_n(T_n^{i}(\omega^1)) - i_1 \\
\label{l2}L_2 &=& \sum_{i=0}^{s_2} l_n(T_n^{i}(\omega^4)) - i_2
\end{eqnarray}
By (P5), we have
\begin{eqnarray*}
A^{E}(L_1, T^{i_1}(\omega^1)) &=& A^{E}_n(T_n^{s_1-1}(\omega^1)) \cdots A^{E}_n(T_n(\omega^1))B^{E}_n(\omega^1) \\
A^{E}(-L_2, \omega^1)^{-1} &=& A^{E}_n(T_n^{s_2}(\omega^4)) \cdots A^{E}_n(T_n(\omega^4)) B^{E}_n(\omega^4) 
\end{eqnarray*}
Denote
\begin{eqnarray*}
G^{E}_1 &=& A^{E}(L_1, T^{i_1}(\omega^1))   \\
G^{E}_2 &=& A^{E}(-L_2, \omega^1)^{-1} \\
C(E) &=& C^{E}_n(\omega^1) = A^{E}(i_1, \omega^1)
\end{eqnarray*}

By  \eqref{B E hyperbolic}  we know that $G^{E}_1, G^{E}_2 \in SL(2,\R) \setminus SO(2,\R)$, and by \eqref{lower bound of norm 2} and \eqref{l1},\eqref{l2} we have
\begin{eqnarray}
\label{ge1} \norm{G^{E}_1} &\geq& e^{(L_1+i_1)\chi_{n+1}}  \\
\label{ge2}\norm{G^{E}_2} &\geq& e^{(L_2+i_2)\chi_{n+1}} 
\end{eqnarray}
Denote
\begin{eqnarray*}
u^1 = u(G^{E}_1), \hspace{1 cm}s^1 = s(G^{E}_1) \\
u^2 = u(G^{E}_2), \hspace{1 cm}s^2 = s(G^{E}_2)
\end{eqnarray*}
Then by \eqref{C1 bound of u n},\eqref{C1 bound of s n} and \eqref{item 1}, we have
\begin{eqnarray*}
|s^1 - s(B^{E}_n(\omega^1))|_{\R/\pi\Z}, |u^2 - u(B^{E}_n(\omega^2))|_{\R/\pi\Z} \leq \zeta_n <\bar{\lambda}_n^{-c}
\end{eqnarray*}

Since $E \notin J$, by \eqref{jn2} we have either
\begin{eqnarray}
\label{first alternate 1}B^{E}_n(\omega^1), B^{E}_n(\omega^2) \in SL(2,\R)\setminus SO(2,\R) \\
\mbox{ and }
\label{first alternate}\angle( R_{\frac{\pi}{2} - s(B^{E}_n(\omega^1))}  C(E) R_{u(B^{E}_n(\omega^2))} \begin{bmatrix} 1 \\ 0 \end{bmatrix}, \begin{bmatrix} 0 \\ 1 \end{bmatrix}  ) \geq \kappa_n
\end{eqnarray}
or
\begin{eqnarray*}
B^{E}_n(\omega^1) \in SO(2,\R) \mbox{ or } B^{E}_n(\omega^2) \in SO(2,\R)
\end{eqnarray*}
The second alternate contradicts \eqref{B E hyperbolic}. Indeed, we can apply \eqref{B E hyperbolic} to $E$, $\omega_4$, $q = s_2-1$ and $r= s_2$; then again apply to $q=s_2$ and $r = s_2+1$.
Thus we have \eqref{first alternate 1} and \eqref{first alternate}.

By \eqref{item 3}, we have $\kappa_n \geq \bar{\lambda}_n^{- \gamma'}$.By $c > \gamma'$ and Lemma \ref{C0 close then angle close} applied to $u^2, s^1, u(B^{E}_n(\omega^2)), s(B^{E}_n(\omega^1))$, when  $\bar{\lambda}_0$ is bigger than some absolute constant, we can ensure that
\begin{eqnarray}
\label{trans}
\angle( R_{\frac{\pi}{2} - s^1}  C(E) R_{u^2} \begin{bmatrix} 1 \\ 0 \end{bmatrix}, \begin{bmatrix} 0 \\ 1 \end{bmatrix}  ) > \kappa_n - c_5\bar{\lambda}_n^{-c} \geq \frac{1}{2} \kappa_n
\end{eqnarray}

We distinguish two cases:

(1)If we have $|\theta(g) - u^2|_{\R/\pi\Z} < \frac{1}{10}c_5^{-1}\kappa_n$.

Then by \eqref{trans} and Lemma \ref{C0 close then angle close},
we have 
\begin{eqnarray*}
\angle(R_{\frac{\pi}{2}-s^1}C(E)\begin{bmatrix}\cos\theta(g) \\ \sin\theta(g) \end{bmatrix} , \begin{bmatrix} 0\\ 1 \end{bmatrix} ) = \angle(R_{\frac{\pi}{2}-s^1}C(E)R_{\theta(g)}\begin{bmatrix}1 \\ 0 \end{bmatrix} , \begin{bmatrix} 0\\ 1 \end{bmatrix} ) > \frac{1}{4}\kappa_n
\end{eqnarray*}
 In this case, when $n$ is larger than some constant depending only on $c_2$, we have
\begin{eqnarray}
\norm{A^{E}(L_1+i_1, \omega^1)g} &=& \norm{A^{E}(L_1, T^{i_1}(\omega^1))A^{E}(i_1, \omega^1)g} \nonumber \\
&=& \norm{G^{E}_1 C(E)g}  \nonumber \\
&=& \norm{R_{u^1} \begin{bmatrix}\norm{G^{E}_1} & 0 \\ 0 & \norm{G^{E}_1}^{-1} \end{bmatrix} R_{\frac{\pi}{2}-s^1} C(E)g}  \nonumber \\
&\gtrsim&c_{5}^{-1}C^{-1} e^{(L_1+t_1)\chi_{n+1}} \kappa_n\norm{g}  \nonumber  \mbox{ ( by \eqref{ge1} )}\\
\label{last ineq}&\geq& c_{5}^{-1}C^{-1} e^{(L_1+t_1)\chi_{n+1} - \gamma' \chi_n \inf l_n} \norm{g} \mbox{ ( by \eqref{item 3} )}
\end{eqnarray}
Since $\omega \in \Omega_n$, we have
\begin{eqnarray}
\label{s1}s_1 &\geq& 2N_n^{\frac{1}{2}}
\end{eqnarray}
By \eqref{l1} it is clear that
\begin{eqnarray}
\label{tmp 2 1} s_1\sup l_n \geq L_1 \geq s_1 \inf l_n - i_1
\end{eqnarray}
Then by \eqref{tmp 2 1}, \eqref{s1}, for all large $n$ we have
\begin{eqnarray}
\label{tmp 2 3} L_1 &\geq& \frac{2}{3}s_1 \inf l_n 
\end{eqnarray}
Moreover by \eqref{tmp 2 4} and \eqref{hat chi chi n}, we have 
\begin{eqnarray*}
\chi_{n+1} \geq c_2^{-1}\hat{\chi} \geq c_2^{-1}\chi_{n}
\end{eqnarray*}
Hence by \eqref{last ineq} and \eqref{s1} we have for all sufficiently large $n$ that
\begin{eqnarray*}
\norm{A^{E}(L_1+i_1, \omega^1)g} \gtrsim c_{5}^{-1}C^{-1} e^{\frac{1}{2}L_1 \chi_{n+1}}\norm{g}
\end{eqnarray*}

By \eqref{s1} and \eqref{item 4}, for $n$ sufficiently large we have
\begin{eqnarray}
\label{tmp 2 5}s_1 > 12C"M_0c_2 \geq12M_nc_2
\end{eqnarray}
where $C"$ is given by Lemma \ref{lemma choosing the parameters}.

Thus
\begin{eqnarray*}
\norm{A^{E}(L_1+i_1+t_1, \omega)h} &=& \norm{A^{E}(L_1+i_1, \omega^1)g} \\
&\gtrsim& c_{5}^{-1}C^{-O(1)}e^{\frac{1}{2}L_1\chi_{n+1}-\sup l_n\hat{\chi}}\norm{h} 
\end{eqnarray*}
By \eqref{tmp 2 3},\eqref{tmp 2 4} and \eqref{tmp 2 5}, we have
\begin{eqnarray*}
\frac{1}{2}L_1\chi_{n+1}-\sup l_n\hat{\chi} &\geq& (\frac{1}{3}s_1\inf l_n -c_2 \sup l_n)\chi_{n+1} \\
&\geq& \frac{1}{4}s_1\inf l_n\chi_{n+1} 
\end{eqnarray*}
By $1\leq i_1\leq K$, \eqref{tmp 2 2},\eqref{tmp 2 1}, \eqref{ratios of time intervals}, \eqref{item 2} and \eqref{item 4}
\begin{eqnarray*}
s_1\inf l_n\chi_{n+1} &\geq& \frac{1}{M_n}s_1 \sup l_n \chi_{n+1} \\
&\geq&\frac{1}{O(M_n)}(L_1+i_1+t_1)\chi_{0} \geq \frac{1}{O(M_0)}(L_1+i_1+t_1) \chi_0 \\
\end{eqnarray*}
Thus we have
\begin{eqnarray*}
\norm{A^{E}(L_1+i_1+t_1, \omega)h} \gtrsim c_{5}^{-1}C^{-1}e^{\frac{1}{O(M_0)}(L_1+i_1+t_1)\chi_0}\norm{h}
\end{eqnarray*}
This contradicts \eqref{poly growth} when $n$ is large.

(2)If we have $|\theta(g) - u^2|_{\R/\pi\Z}  \geq \frac{1}{10}c_{5}^{-1}\kappa_n$

 Since
\begin{eqnarray}
u^2 = s(A^{E}(-L_2, \omega^1))
\end{eqnarray}
Similar computations shows that for all sufficiently large $n \geq 0$ we have
\begin{eqnarray*}
\norm{A^{E}(-L_2, \omega^1)g} &=& \norm{(G^{E}_2)^{-1} g} \\
&\gtrsim& c_{5}^{-1}C^{-1}e^{L_2\chi_{n+1}} \kappa_n\norm{g}  \\
&\gtrsim& c_{5}^{-1}C^{-1}e^{\frac{1}{2}L_2 \chi_{n+1}}\norm{g}
\end{eqnarray*}
and we can reach a contradiction in a way similar to (1).
This proves the statement in the lemma.

\end{proof}


\section{Area of the spectrum and the proof of Theorem \ref{main theorem}} \label{Area of the spectrum and the proof of the main theorem}
To prove Theorem \ref{real main theorem}, and as a consequence, Theorem \ref{main theorem},
it remains to estimate the measure of $\overline{J}$, where $J$ is defined in \eqref{def of J} in Section \ref{choosing the parameters}.

\begin{nota}
For any $n\geq 0$, any $\alpha \in \mathcal{A}_n$ such that $\alpha = \omega_i\omega_{i+1}\cdots \omega_{l_n(\omega)-1} $ for some $1\leq i\leq K$ and $\omega \in \Delta_{(i)}$, for each $0 \leq m \leq n-1$, we define
\begin{eqnarray*}
hd_{m}(\alpha) &=& \omega_i\omega_{i+1} \cdots \omega_{l_{m}(\omega)-1}  \\
rr_{m}(\alpha) &=& \tilde{\omega}_j \tilde{\omega}_{j+1} \cdots \tilde{\omega}_{l_{m}(\tilde{\omega})-1} 
\end{eqnarray*}
Here $\tilde{\omega} = T_{m}^{-1}T_n(\omega)$ such that $\tilde{\omega} \in \Delta_{m}\bigcap\Delta_{(j)}$ for some $1\leq j \leq K$. We can verify by (P2) that $hd_{m}(\alpha), rr_{m}(\alpha)$ are respectively prefix and suffix of $\alpha$, and belong to $\mathcal{A}_{m}$.
\end{nota}

The following estimate is essentially proved in \cite{ADZ} ( see also \cite{WZ} ) by explicit calculations. Here we give a sketched proof.
\begin{lemma} \label{lemma C1}
There exists $C_1> 0$ such that for all sufficiently large $\lambda$ the following is true.
For all $\alpha, \beta \in \mathcal{A}_0$, any $1\leq j \leq K$, any $E \in I_{E_0}$ we have $A^{E}_{\alpha}, A^{E}_{\beta} \in SL(2,\R) \setminus SO(2,\R)$. As functions from $I_{E_0}$ to $\R / \pi\Z$, $E \mapsto s(A^{E}_{\beta})$ and $E \mapsto u(A^{E}_{\alpha})$ are $C^1$. Moreover, consider $E \mapsto R_{\frac{\pi}{2}-s(A^{E}_{\beta})}C^E_j R_{u(A^{E}_{\alpha})}\begin{bmatrix} 1 \\ 0 \end{bmatrix}$ as a function from $I_{E_0}$ to $\mathbb{P}\R^2$, we have $$ |\partial_E( R_{\frac{\pi}{2}-s(A^{E}_{\beta})}C^E_j R_{u(A^{E}_{\alpha})}\begin{bmatrix} 1 \\ 0 \end{bmatrix})| > C_1  $$
\end{lemma}
\begin{proof}
Since for any $\alpha \in \mathcal{A} \setminus \{\alpha_0\}$, any $E \in I_{E_0}$, we have $|tr(A^{E}_{\alpha})| \geq \lambda - H$. When $\lambda > H+2$, we have $A^{E}_{\alpha} \in SL(2,\R) \setminus SO(2,R)$.

Denote 
\begin{eqnarray*}
v = \begin{bmatrix} 1 \\ 0 \end{bmatrix}, s_{\beta}(E) = s(A^{E}_{\beta}), u_{\alpha}(E) = u(A^{E}_{\alpha})
\end{eqnarray*}
It well-known that under the condition of the lemma, $s_{\beta}, u_{\alpha}$ are $C^1$.

We have
\begin{eqnarray}
\label{der rcr} &&\partial_E(R_{\frac{\pi}{2} - s_{\beta}(E)}C^E_jR_{u_{\alpha}(E)} v )  \\
  &=& \partial_ER_{\frac{\pi}{2} - s_{\beta}(E)} ( C^E_j R_{u_{\alpha}(E)} v) + DR_{\frac{\pi}{2} - s_{\beta}(E)} (C^E_j R_{u_{\alpha}(E)} v) \partial_E C^E_j (R_{u_{\alpha}(E)} v)   \nonumber \\
  &&+ D(R_{\frac{\pi}{2} - s_{\beta}(E)}C^E_j)(R_{u_{\alpha}(E)} v) \partial_ER_{u_{\alpha}(E)}( v ) \nonumber
\end{eqnarray}
Here and the following, the derivatives of varies functions from $E$ to $\mathbb{P}\R^2$ are interpreted through identifying $\R/\pi\Z$ with $\mathbb{P}\R^2$ as 
\begin{eqnarray*} \theta \in \R/\pi\Z \mapsto \R\begin{bmatrix} \cos\theta \\ \sin \theta \end{bmatrix} \in \mathbb{P}\R^2 \end{eqnarray*}
Since $SL(2,\R)$ act $\mathbb{P}R^2$ through smooth, orientation preserving diffeomorphisms, for any $M \in SL(2,\R)$, any $\psi \in \mathbb{P}\R^2$, we have
\begin{eqnarray*}
DM(\psi) > 0
\end{eqnarray*}
It is shown in \cite{ADZ} ( see also \cite{WZ}) that 
\begin{eqnarray*}
\partial_ER_{\frac{\pi}{2} - s_{\beta}(E)} (\phi) \leq 0 \\
\partial_ER_{u_{\alpha}(E)}( \phi ) \leq 0
\end{eqnarray*}
for all $\phi \in \mathbb{P}\R^{2}$.
So the first term and the last term in \eqref{der rcr} is non-positive.

Since for any $\phi \in \mathbb{P}\R^2$, we have
\begin{eqnarray*}
DR_{\frac{\pi}{2} - s_{\beta}(E)}(\phi) = 1
\end{eqnarray*}
it remains to check that $ \partial_E C^E_j (R_{u_{\alpha}(E)} v)$ is uniformly bounded by a negative constant for all choice of $\alpha \in \mathcal{A}$,$1 \leq j \leq K$ and $E \in I_{E_0}$.

Denote $\phi = R_{u_{\alpha}(E)}v$.
It is well-known that for any $\psi \in \mathbb{P}\R^2$, we have
\begin{eqnarray*}
\partial_{E}C^{E}(\psi) \leq 0
\end{eqnarray*}
Then we have 
\begin{eqnarray*}
\partial_{E}C^{E}_{j}(\phi) = \sum_{i = 0}^{j-1} DC^{E}_i(C^{E}_{j-i}(\phi)) \partial_{E}C^{E}(C^{E}_{j-i-1}(\phi)) \leq DC^{E}_{j-1}(C^{E}(\phi))\partial_{E}C^{E}(\phi)
\end{eqnarray*}
When $\lambda$ is sufficiently large, $u_{\alpha}(E)$ can be made arbitrarily close to $0$ for all $E \in I_{E_0}$.
Then we can ensure that $\phi$ is close to $v$ so that
\begin{eqnarray*}
\partial_{E}C^{E}_{k}(\phi) \leq \frac{1}{2}DC^{E}_{k-1}(C^{E}(\phi))\partial_{E}C^{E}(v) 
\end{eqnarray*}
Straight-forward computation shows that the right hand is strictly negative. This completes the proof.

\end{proof}

Now we are going to show that the Lebesgue measure of $\overline{J}$ is small.

\begin{lemma} \label{lemma the closure of the spectrum}
 For any $0 < \gamma < \gamma' < \frac{1}{4}$, any $\gamma' < c < 2-3\gamma'$, there exists $C_3 > 0$, such that the following is true. For $\lambda$ sufficiently large, we define $J_n$ and parameters $ \bar{\lambda}_n, \zeta_n, \chi_n, M_n, N_n, \kappa_n$ that satisfy the conclusions in Lemma \ref{lemma choosing the parameters} with $\gamma, \gamma', c$. Define $J$ by \eqref{def of J n}. Then $ Leb(\overline{J}) \leq C_3 \lambda^{-\gamma}$
\end{lemma}
\begin{proof}
By Lemma \ref{lemma choosing the parameters}, for all sufficiently large $\lambda$, we have \eqref{ratios of time intervals}  to \eqref{lower bound of norm} and \eqref{item 1} to \eqref{item 4} for all $n \geq 0$.

For any $n \geq 1$, for any $\alpha \in \mathcal{A}_n$, for any $E \in I_{E_0}$, if $\alpha = \omega_i \omega_{i+1} \cdots \omega_{l_n(\omega)-1}$ for some $\omega \in \Delta_{n} \bigcap \Delta_{(i)}$ and $1\leq i \leq K$, by definition we have
 \begin{eqnarray*} 
 A^{E}(\alpha) &=& B^{E}_n(\omega) \\
 A^{E}(hd_{n-1}(\alpha)) &=& B^{E}_{n-1}(\omega)\\
 A^{E}(rr_{n-1}(\alpha)) &=& B^{E}_{n-1}(T_{n-1}^{r_{n-1}(\omega)-1}(\omega))
 \end{eqnarray*}
 
Then for any $\alpha, \beta \in \mathcal{A}_n$, for any $1 \leq i \leq K$, for any $E \in J(\alpha, \beta, i, \kappa_n) \setminus \bigcup_{0 \leq j \leq n-1}J_j$, by \eqref{B E hyperbolic} we have
\begin{eqnarray*}
A^{E}(hd_{j}(\beta)), A^{E}(rr_{j}(\alpha)) \in SL(2,\R) \setminus SO(2,\R), \forall 0 \leq j \leq n-1
\end{eqnarray*}
For any $n \geq 1$, any $\alpha, \beta \in \mathcal{A}_n$, any $E \in I_{E_0} \setminus \bigcup_{0 \leq m \leq n-1} J_m$, by \eqref{C1 bound of u n}
 we know that
\begin{eqnarray}
\label{prev1}|u(A^{E}(\alpha)) - u(A^{E}(rr_{n-1}(\alpha)))|_{\R/\pi\Z} &<& \zeta_{n-1}<\bar{\lambda}_{0}^{-2^{n-1}c} 
\end{eqnarray}
by \eqref{C1 bound of s n}
\begin{eqnarray}
\label{prev2}|s(A^{E}(\beta)) - s(A^{E}(hd_{n-1}(\beta)))|_{\R/\pi\Z}  &<& \bar{\lambda}_{0}^{-2^{n-1}c}
\end{eqnarray}

Moreover by Lemma \ref{C0 close then angle close}, \eqref{prev1}, \eqref{prev2} for all $1 \leq m \leq n$, we see that for any $E \in J(\alpha, \beta, i, \kappa_n) \setminus \bigcup_{0 \leq j \leq n-1}J_j$ we have the following
\begin{eqnarray*}
\bar{\lambda}_0^{-\gamma}
&\geq& \kappa_n \geq \angle(  R_{\frac{\pi}{2}-s(A^E(\beta))}  C^E_i R_{u (A^E(\alpha))}\begin{bmatrix} 1 \\ 0 \end{bmatrix},\begin{bmatrix} 0 \\ 1 \end{bmatrix} ) \\
&\geq& 
\angle(  R_{\frac{\pi}{2}-s(A^E(hd_{n-1}(\beta)))}  C^E_i R_{u (A^E(rr_{n-1}(\alpha)))}\begin{bmatrix} 1 \\ 0 \end{bmatrix},\begin{bmatrix} 0 \\ 1 \end{bmatrix} ) - c_5 \bar{\lambda}_{0}^{-2^{n-1}c} \\
&&\cdots \\
&\geq& \angle(R_{\frac{\pi}{2} - s(A^{E}_{hd_{0}(\beta)})}C^{E}_iR_{u(A^{E}_{rr_{0}(\alpha)})}\begin{bmatrix} 1 \\ 0 \end{bmatrix},\begin{bmatrix} 0 \\ 1 \end{bmatrix}) - c_5(\bar{\lambda}_0^{-c} + \cdots + \bar{\lambda}_{0}^{-2^{n-1}c})
\end{eqnarray*}
Then 
\begin{eqnarray*}\angle(R_{\frac{\pi}{2} - s(A^{E}_{hd_{0}(\beta)})}C^{E}_iR_{u(A^{E}_{rr_{0}(\alpha)})}\begin{bmatrix} 1 \\ 0 \end{bmatrix},\begin{bmatrix} 0 \\ 1 \end{bmatrix})  \leq c_5(\bar{\lambda}_0^{-c} + \cdots + \bar{\lambda}_{0}^{-2^{n-1}c}) + \bar{\lambda}_0^{-\gamma}
\end{eqnarray*}
Denote $\theta = \sup_{n \geq 0} (c_5(\bar{\lambda}_0^{-c} + \cdots + \bar{\lambda}_{0}^{-2^{n-1}c}) + \bar{\lambda}_0^{-\gamma})$.
Then we have that $E \in J(rr_0(\alpha), hd_0(\beta), j, \theta)$.

Then for all $n \geq 0$, $\alpha, \beta \in \Delta_n$, any $1 \leq i \leq K$ we have the following
\begin{eqnarray*}
(J(\alpha, \beta, i, \kappa_n ) \setminus \bigcup_{0 \leq j \leq n-1}J_j )\subset  \overline{J(rr_0(\alpha), hd_0(\beta), i, \theta)}
\end{eqnarray*}
Take the unions of the above expression for all $n \geq 0$, $\alpha, \beta \in \Delta_n$, all $1 \leq i \leq K$, we obtain 
\begin{eqnarray*}
J \subset \bigcup_{\alpha, \beta \in \mathcal{A}_0, 1\leq j\leq K} \overline{J(\alpha, \beta, j, \theta)}
\end{eqnarray*}
The right hand side is a closed set and by Lemma \ref{lemma C1} and \eqref{C1 upper bound of critical matrices}, it is of measure $O(\theta)$.
Since $c > \gamma$, then there exists a constant $Q > 0$ depending only on $\gamma, \gamma', c$ such that $\theta < Q\lambda^{-\gamma}$ for all $\lambda$ sufficiently large.
This concludes the proof.
\end{proof}

\begin{proof}[Proof of Theorem \ref{real main theorem}]
For any $\gamma \in (0, \frac{1}{4})$, we can choose $\gamma' \in (\gamma, \frac{1}{4})$ and $c \in (\gamma', 2-3\gamma')$. Then when $\lambda$ is sufficiently large, Theorem \ref{real main theorem}
follows from Lemma \ref{lemma spectrum} and Lemma \ref{lemma the closure of the spectrum}. When $\lambda$ is small, we use the trivial bound $Leb(\Sigma_{v} \bigcap I_{E_0}) \leq Leb(I_{E_0}) \leq 2H$. After possibly enlarging $Q$, we obtain Theorem \ref{real main theorem}.
\end{proof}

\section{Proof of Theorem \ref{positive measure spectrum for large couplings theorem} and Theorem \ref{unbounded repetitions theorem}} \label{folklore theorems}

\subsection{Proof of Theorem \ref{positive measure spectrum for large couplings theorem}}
The construction of the required subshift follows closely the proof of Theorem 1 in \cite{ADZ}. We refer to \cite{ADZ} for some relevant lemmata.
Without loss of generality, let us assume that 
$B$ is a countably infinite set of potentials and $0 < \epsilon < 1$.
We will inductively define collections of finite words $S_n$, subshifts $\Omega_n$, closed subsets $\Sigma_{n,m}$ for $1\leq n \leq m$.

For $n=1$, we define
\begin{eqnarray}
S_1 &=& \{1,\cdots, k\}
\end{eqnarray}
We define $\Omega_1$ to be the two-sided infinite concatenations of the words in $S_1$.
We now pick any element $v_1 \in B$.
For each word $w \in S_1$, we denote the spectrum of the periodic potential associated to $v_1$ and $w$ by $\Sigma_{1,1}(w)$, and define
\begin{eqnarray}
\Sigma_{1,1} &=& \bigcup_{w \in S_1} \Sigma_{1,1}(w)
\end{eqnarray}

Assume $S_n$, $\Omega_n$, $\Sigma_{i,n}, \forall 1\leq i \leq n$ are constructed. We denote
\begin{eqnarray*}
S_n &=& \{w_{n,1},w_{n,2}, \cdots, w_{n,k_n} \}
\end{eqnarray*}
For any given integer $N_n \geq 1$, we define
\begin{eqnarray*}
S_{n+1} = \{ w_{n,1}w_{n,2}\cdots w_{n,k_n}w_{n,k}^{l} ; 1\leq k\leq k_n, N_n\leq l < N_n + N_n^{\frac{1}{2}\epsilon} \}
\end{eqnarray*}
and define $\Omega_{n+1}$ to be the two-sided infinite concatenation of the words in $S_{n+1}$.
It is direct to see that $\Omega_{n+1} \subset \Omega_{n}$.

We pick any element $v_{n+1} \in B \setminus \{v_1, \cdots, v_n\}$.
For each $1\leq i \leq n+1$,  for each $w \in S_{n+1}$, we denote the spectrum of the periodic potential associated to $v_i$ and $w$ by $\Sigma_{i,n+1}(w)$, and denote
\begin{eqnarray}
\Sigma_{i,n+1}  = \bigcup_{w \in S_{n+1}} \Sigma_{i,n+1}(w)
\end{eqnarray}
It is clear that $Leb(\Sigma_{n+1}) > 0$.
By a slightly modified version of Lemma 1 in \cite{ADZ}, we can choose a positive integer $N_{n}$ depending only on  $S_n$, $\Omega_n$, $\Sigma_{i,n}$ such that the following is true.
\begin{eqnarray}
Leb(\Sigma_{i,n} \setminus \Sigma_{i,n+1}) < Leb(\Sigma_{i,i}) 2^{-(n+1)}
\end{eqnarray}
for any $1 \leq i \leq n$.
We define $\Omega = \bigcap_n \Omega_n$.
For each $v \in B$, denote the spectrum associated to $\Omega$ and $v$ by $\Sigma$. For some $i \in \N$, we have $v = v_i$.
Then following \cite{ADZ}, we have
\begin{eqnarray}
\Sigma \supseteq \lim \sup_{n \to \infty} \Sigma_{i,n} 
\end{eqnarray}
Then by the same reasoning in \cite{ADZ}, we have $Leb(\Sigma) > \frac{1}{2}Leb(\Sigma_{i,i}) > 0$.
Following the proof of Lemma 2 in \cite{ADZ}, we can show that $\Omega$ is minimal and aperiodic.

It remains to show that when $N_n$ are properly chosen, we can ensure that $\Omega$ has required complexity function.

For any $n \geq 0$, define
\begin{eqnarray}
   M_n = \min \{ |w| ;  w \in S_n \}, 
   P_n = \max \{ |w| ; w \in S_n \}
\end{eqnarray}
It is direct to see that
\begin{eqnarray}
M_{n+1} &\geq& N_n M_n \\
P_{n+1} &\leq& (N_n + N_n^{\frac{1}{2}\epsilon})P_n \\
  S_{n+1} &=& N_n^{\frac{1}{2}\epsilon}|S_n|
\end{eqnarray}
Hence for any $n \geq 0$
\begin{eqnarray}
S_n &\lesssim& M_n^{\frac{1}{2}\epsilon} \\
\frac{P_{n+1}}{M_{n+1}} &\leq& (1 + N_n^{-1+\frac{1}{2}\epsilon})\frac{P_n}{M_n}
\end{eqnarray}
From the construction, we see that we can also ensure that
\begin{eqnarray}
\sum_{n\geq 0}^{\infty} N_{n}^{-1 + \frac{1}{2}\epsilon} < \infty
\end{eqnarray}
Then there exists $C > 0$ such that for any $n \geq 0$, we have
\begin{eqnarray}
P_n \leq C M_n
\end{eqnarray}

For any $L \in \N$, there exists $n \in \N$ such that $M_n \leq L < M_{n+1}$.
For any word $w$ of length $L$, there exists two words $w_1,w_2 \in S_{n+1}$, such that $w$ is a subword of the concatenation $w_1w_2$ and is not a subword of $w_1$. Assume
\begin{eqnarray}
w_1 = w_{n,1} \cdots w_{n,k_n} w_{n,i}^{l} \\
w_2 = w_{n,1} \cdots w_{n,k_n} w_{n,j}^{m} 
\end{eqnarray}
We have four possibilities:

(1) $w$ does not intersect $w_{n,1} \cdots w_{n,k_n} $. Then $w$ is a subword of $ w_{n,j}^{m}$. Then there are at most $|S_n|L$ possible choices of $w$;

(2) $w$ contains $w_{n,1} \cdots w_{n,k_n} $. Then $w$ is the concatenation of a suffix of $w_{n,i}^{l}$ (possibly empty),$w_{n,1} \cdots w_{n,k_n} $ and a prefix of $w_{n,j}^{m}$. In this case, there are at most $|S_n|^2L$ possible choices of $w$;

(3) $w$ intersect both $w_{n,1} \cdots w_{n,k_n}$ and $w_{n,j}^{m}$. Then $w$ is determined by a prefix of $w_{n,j}^{m}$ of length at most $L$. There are at most $|S_n|L$ possible choices of $w$;

(4) $w$ is contained in $w_{n,1} \cdots w_{n,k_n}$. Then there are at most $P_n|S_n|$ possibilities. Since $P_n \leq C M_n \leq C L$, we have at most $CL|S_n|$ possibilities.

Combining all three cases, we have
\begin{eqnarray}
p(L) \leq |S_n|L + |S_n|^2L + |S_n|L + C|S_n|L\lesssim L^{1+\epsilon}
\end{eqnarray}

This proves the theorem.

\subsection{Proof of Theorem \ref{unbounded repetitions theorem}}
Fix $E \in (-2+v(i), 2+v(i))$, then $A^{E}_i$ is an elliptic matrix.
Assume to the contrary that $E \notin \Sigma_{v}$.
Then we can take an open interval neighbourhood of $E$, denoted by $J$, such that 
$ J \subset (-2+v(i), 2+v(i)) \bigcap \Sigma_{v}^{c}$.
By Theorem \ref{johnson} the cocycle $A^{E}$ over $\Omega$ is Uniformly Hyperbolic. Thus we can define stable, unstable directions, denoted respectively by $s,u : \Omega \to \mathbb{P} \R^2$. After possibly reducing $J$, we can assume that for any $E' \in J$, we have $s(E'),u(E') : \Omega \to  \mathbb{P} \R^2$, and for any $\omega \in \Omega$, the function $s(\cdot, \omega), u(\cdot, \omega): J \to \mathbb{P} \R^2$ are $C^1$ ( in fact analytic )  and the $C^1$ norm of these functions are bounded uniform in $\omega \in \Omega$. We take any $\omega \in \Omega$ such that $\omega_0 = \cdots = \omega_{N-1} = i$, where $N$ will be chosen to be large. Denote $\omega' = T^{N}(\omega)$. Then $s(E', \omega') = (A^{E'}_i)^N s(E',\omega)$ for all $E' \in J$. Straightforward calculation shows that the $C^1$ norm of $s(\cdot, \omega')$ will be $\Theta(N)$. When $N$ is large, we have a contradiction. Hence $E \in \Sigma_{v}$. This proves the theorem.

\end{document}